\documentclass[journal,twoside,web]{ieeecolor}
\usepackage{generic}
\usepackage{generic,enumerate}
\usepackage{cite}
\usepackage{amsmath,amssymb,amsfonts}
\usepackage{multirow}
\usepackage{array}
\usepackage{dsfont}
\allowdisplaybreaks
\usepackage{mathtools,comment}
\usepackage{algorithm}
\usepackage{algpseudocode} 
\usepackage{graphicx}
\usepackage{tikz}
\usetikzlibrary{%
    decorations.pathreplacing,%
    decorations.pathmorphing%
}
\usetikzlibrary{arrows.meta}
\usetikzlibrary{automata,positioning}
\usepackage{textcomp}
\usepackage{xcolor}
\let\originalleft\left
\let\originalright\right
\renewcommand{\left}{\mathopen{}\mathclose\bgroup\originalleft}
\renewcommand{\right}{\aftergroup\egroup\originalright}
\newcommand{\bE}{\mathbb{E}}

\newcommand{\cS}{\mathcal{S}}
\newcommand{\cD}{\mathcal{D}}
\newcommand{\bN}{\mathbb{N}}
\newcommand{\cN}{\mathcal{N}}

\newcommand{\cU}{\mathcal{U}}
\newcommand{\cX}{\mathcal{X}}
\newcommand{\ust}{^{\star}}
\newcommand{\ub}{^{(\beta)}}
\newcommand{\ubn}{^{(\beta_n)}}
\newcommand{\bB}{\mathbb{B}}
\newcommand{\bR}{\mathbb{R}}
\newcommand{\bZ}{\mathbb{Z}}
\newcommand{\te}{\theta}

\newcommand{\cI}{\mathcal{I}}
\newcommand{\cP}{\mathcal{P}}
\newcommand{\bP}{\mathbb{P}}

\newcommand{\cB}{\mathcal{B}}
\newcommand{\cA}{\mathcal{A}}

\newcommand{\cF}{\mathcal{F}}
\newcommand{\cH}{\mathcal{H}}

\newcommand{\D}{\Delta}

\newcommand{\vp}{\varphi}

\newcommand{\lf}{\left}
\newcommand{\rt}{\right}

\newcommand{\g}{\gamma}
\newcommand{\ve}{\varepsilon}
\newcommand{\pih}{\pi^{(hyb.)}}
\newcommand{\ta}{\tau_{\cA(\ve)}}

\newcommand{\nal}[1]{\begin{align*}#1\end{align*}}
\newcommand{\al}[1]{\begin{align}#1\end{align}}
\newtheorem{assumption}{\textbf{Assumption}}

\newtheorem{definition}{Definition}
\newtheorem{theorem}{Theorem}
\newtheorem{proposition}{Proposition}
\newtheorem{lemma}{Lemma}
\newtheorem{remark}{Remark}
\newif\ifuseRomanappendices
\useRomanappendicesfalse

\def\BibTeX{{\rm B\kern-.05em{\sc i\kern-.025em b}\kern-.08em
    T\kern-.1667em\lower.7ex\hbox{E}\kern-.125emX}}
\begin{document}

\title{Optimal Scheduling for Remote State Estimation over Hybrid Channels
}

\author{Manali Dutta, Rahul Singh, and Shalabh Bhatnagar
\thanks{Manali Dutta and Shalabh Bhatnagar are with the Department of Computer Science
	and Automation, Indian Institute of Science, Bangalore 560012,
	Karnataka, India. E-mail: manalidutta@iisc.ac.in, shalabh@iisc.ac.in. }
\thanks{E-mail: rahulsingh0188@gmail.com.}
}

\maketitle
\thispagestyle{empty}

\begin{abstract}
	We study optimal scheduling for remote state estimation over a network with two heterogeneous communication channels: a fast but unreliable channel and a slow but reliable channel. To capture temporal correlations in packet losses, we model the unreliable channel as a Gilbert-Elliott (GE) channel. The remote estimation setup consists of a source, a sensor, and a remote estimator. The source evolves as a discrete-time autoregressive (AR) process, and the sensor decides at each time whether to use the fast unreliable channel or the slow reliable channel. We formulate the scheduling problem faced by the sensor as a Markov decision process (MDP) with a continuous state-space and consider minimizing the infinite horizon average cost criterion, where the cost consists of the squared estimation error and the transmission energy consumed. We establish the existence of an optimal stationary policy. We then characterize the structure of an optimal policy, and show that it has a threshold structure with respect to the estimation error. An optimal policy chooses from amongst the two channels based on whether the error exceeds certain thresholds, where the threshold value depends upon the GE channel state. When the system parameters are unknown, we propose an actor-critic (AC) learning algorithm that exploits the threshold structure of an optimal policy. Numerical results demonstrate that the proposed AC algorithm learns the policy structure effectively and achieves performance close to that of the optimal policy computed using the relative value iteration (RVI). 
\end{abstract}


\section{Introduction}

In a networked control system (NCS), the control loop is closed over a shared communication channel~\cite{zhang2019networked}. A sensor measures a physical process, encodes each measurement into a packet, and sends it to a remote estimator via a communication channel, whose output is used for monitoring or feedback control. Such setups arise in industrial automation, cyber-physical infrastructure, aircraft, and remote health monitoring~\cite{zhang2019networked}. In an NCS, the sensor faces a fundamental trade-off: transmitting a packet consumes energy and occupies the network, while not transmitting allows the estimation error to grow. The sensor must therefore decide \emph{when} to transmit. Moreover, when more than one channel is available, it must also decide \emph{which} channel to use. Since the estimator can act only on received information, the quality of the communication channel limits the achievable estimation and control performance~\cite{schenato2007foundations, wang2023review}.

The earliest studies modeled the channel in the simplest way where each packet is either delivered or dropped, with the drops occurring independently over time. This is the independent and identically distributed (i.i.d.) Bernoulli packet drop model. The seminal work~\cite{sinopoli2004kalman} shows that, under this model, there is a critical delivery probability below which the expected estimation error covariance diverges. Subsequent work extended this idea to control~\cite{imer2006optimal} and to communication channels with random delays~\cite{schenato2008optimal}. Once a cost is attached to each transmission, the problem becomes that of a scheduling problem. For a first-order source,~\cite{lipsa2011remote} proves that the optimal estimator is ``Kalman-like'' and the optimal transmission policy is of \emph{threshold} type, i.e., the sensor transmits only when the estimation error exceeds a certain threshold. Similar threshold rules are shown for an energy-harvesting sensor~\cite{nayyar2013optimal}, for autoregressive (AR) source~\cite{chakravorty2017fundamental, chakravorty2016remote}, and event-triggered estimation~\cite{wu2013event, molin2017event}.

However, the i.i.d. channel model has no memory, i.e., whether a packet is lost at one time provides no information about future losses. Practical communication channel often behave differently. Losses occur in \emph{bursts}, because the physical cause of a loss for example, congestion, a temporary obstruction, or deep fading, can persist over multiple time steps~\cite{yajnik1999measurement}. Measurement studies confirm that consecutive losses are strongly correlated in time~\cite{yajnik1999measurement, wang1995finite}. It is shown in~\cite{huang2007stability} that under correlated losses the arguments of the i.i.d.\ packet drop channel setting in~\cite{sinopoli2004kalman} no longer apply, and that stability depends on the correlation structure and not only on the mean loss rate. Thus, two links with the same average
loss rate can have very different effects on remote estimation performance. This observation motivates channel models with
memory.

A standard model for bursty packet losses is the Gilbert--Elliott (GE) channel, a two-state Markov chain with a \emph{good} state in
which packets are delivered and a \emph{bad} state in which packets are dropped. The transition probabilities determine the typical length of loss bursts. Originally introduced by Gilbert~\cite{gilbert1960capacity} and Elliott~\cite{elliott1963estimates} to describe burst errors on wired telephone circuits, the GE channel model has since been used far beyond that setting. In NCSs it is used to model an unreliable link with temporal correlation~\cite{qi2016optimal, chakravorty2017structure, farjam2023distributed, dutta2023optimal}. For remote estimation over a single GE channel, the works~\cite{qi2016optimal, chakravorty2019remote} establish threshold-type scheduling policies when the channel state is known to the sensor, and~\cite{dutta2023optimal, farjam2023distributed} do so when it is only partially observed.

The works described above consider a single transmission channel for the sensor. In many systems, however, the sensor can choose among several channels. In~\cite{gao2018remote}, the sensor chooses at each time step between a cheap noisy channel and a costly noiseless channel, both of which are memoryless. It is shown that the optimal policy is not, in general, a simple symmetric threshold in the source state, and restoring such a structure requires an auxiliary side channel that signals the sign of the state. The works~\cite{liu2022remote, liu2023stability} consider the scheduling of multiple sensors over multiple Markov fading channels, and derive necessary and sufficient conditions on the channel and process parameters for the remote estimation system to remain stable. The work~\cite{leong2023stability} designs a bandit algorithm that learns to select among several channels with unknown channel statistics while keeping the remote state estimation system stable. These works treat the available channels either as memoryless~\cite{gao2018remote} or as statistically similar in type~\cite{liu2022remote, liu2023stability,
leong2023stability}. To the best of our knowledge, none characterizes the structure of an optimal scheduling policy over two
channels that differ simultaneously in temporal correlation, delay, and reliability.

In practice, it is useful to combine channels of different kinds, because a fast channel is usually fragile, and a dependable channel is usually slow. Rather than relying on a single channel, practical systems combine multiple channels and moves traffic between them. We refer to this setting as a \emph{hybrid} channel model. In automotive electronics, for
example, FlexRay supports predictable safety-critical messages, Controller Area Network (CAN) supports general communication
despite nondeterministic delays, and automotive Ethernet carries high-bandwidth data~\cite{kim2015gateway}. Production vehicles often connect these
networks through a gateway~\cite{kim2015gateway}. In underwater systems, acoustic links provide long-range reliable communication at low data rates, while optical links provide much higher rates over shorter clear-water ranges. Therefore, so hybrid networks use optical links for bulk data and acoustic links for control~\cite{wang2017design}. Similar pairings appear in unmanned aerial vehicles~\cite{xiao2021survey}, smart-grid measurement systems that combine power-line and fiber links~\cite{galli2011grid} and, of course, in communication systems~\cite{semiari2019integrated, he2024efficient, fang2021embraces}. In all
these examples, links differ simultaneously in reliability, delay, and memory, which is precisely the combination that makes
optimal scheduling nontrivial. To the best of our knowledge, no prior work characterizes an optimal scheduling policy over hybrid channel model in a remote estimation setup.

In this paper, we study remote estimation of a discrete-time AR source over such a hybrid communication architecture. The sensor
observes the source, encodes observations into packets, and transmits them to a remote estimator over a communication channel. At each time, the sensor chooses between a fast but unreliable GE channel and a slow but reliable channel. The unreliable channel is modeled as a GE channel~\cite{gilbert1960capacity}. The reliable channel always delivers, but only after a fixed delay of $d$ time steps, during which it cannot accept a new packet. We minimize the infinite-horizon average cost~\cite{Hernandez2012discrete} consisting of squared estimation error plus transmission energy. A similar hybrid channel model is also studied in~\cite{pan2022age}, but for age of information (AoI), which measures information staleness but does not take into account the realized source value. Moreover, unlike AoI problems~\cite{kadota2018optimizing, lu2018age, bedewy2021low, yates2021age} that typically involve a countable state-space, our problem involves a continuous state-space arising from the AR source dynamics, our cost is unbounded, and our switching rule depends on the source dynamics. The main contributions are as follows.

\begin{enumerate}
	\item We formulate remote estimation of an AR source over a hybrid channel that pairs a fast, unreliable GE channel with a slow, reliable channel. Prior works have considered a single channel~\cite{lipsa2011remote, chakravorty2016remote, wu2019learning, ren2017infinite, qi2016optimal, chakravorty2019remote, dutta2023optimal}, two memoryless channels that differ in cost and reliability but not in delay or memory~\cite{gao2018remote}, or several channels of the same kind~\cite{farjam2023distributed, liu2022remote, liu2023stability, leong2023stability, wang2020whittle}. Our model differs from these settings because the two channels differ simultaneously in temporal correlation, delay, and reliability.

	\item We pose the sensor's problem as a Markov decision process (MDP) under the infinite horizon average cost criterion. It is a challenging problem since (i) the state-space is continuous, (ii) the instantaneous cost function is not bounded, (iii) the unreliable channel has a memory, and (iv) the two channels have different characteristics: deterministic delay of $d$ units and Markovian channel state. We nonetheless establish the existence of an optimal stationary policy.

	\item We characterize the following structure of an optimal policy. For each GE channel state, there is a threshold on the estimation error at which an optimal policy switches channels, and the direction of the switch depends on the channel and AR parameters. When the fast channel is sufficiently reliable, the policy uses it for large errors to correct them quickly. When the fast channel is highly unreliable, the policy switches to the slower reliable channel for large errors because a likely packet drop on the fast channel would be costly. The intuition associated with this structure is discussed in Section~\ref{subsec:intuition}. A switching structure is also reported in~\cite{pan2022age} in the AoI setting of. However, the problem in~\cite{pan2022age} has a countable state space and its threshold depends
    only on channel parameters. In our worl, the estimation error depends on the realized source trajectory, so the thresholds depend on
    both the AR source dynamics and the channel parameters.

	\item The switching rule depends on the AR and channel parameters, which are often unknown to the operator and may vary over time. We therefore design an actor-critic learning algorithm~\cite{konda2003onactor} that exploits the structural properties of an optimal policy. The structural result significantly reduces the complexity of searching for an optimal policy. This is because instead of searching over an infinite-dimensional space of admissible scheduling policies, we consider a parameterized class of policies with a finite number of parameters that captures the structural properties of an optimal policy. These parameters are then tuned using the AC algorithm.~Numerical results show that the proposed AC algorithm learns the predicted structure and adapt across different parameter regimes. Moreover, its performance is observed to be close to that of the optimal policy computed via relative value iteration (RVI)~\cite{Hernandez2012discrete}, which requires full knowledge of the system parameters.
\end{enumerate}

The rest of the paper is organized as follows.
Section~\ref{sec:prob_form} formalizes the system model and states the optimization problem faced by the sensor. Section~\ref{subsec:MDP formulation} casts this problem as an MDP, introduces the associated discounted problem, and describes the value iteration algorithm. Section~\ref{sec:struc_res} proves the existence of an optimal stationary policy and characterizes its threshold structure, first for the discounted problem and then for the average cost problem. Section~\ref{sec:num_res} presents numerical results where we first compute an optimal policy via RVI when the parameters are known and then develop an AC algorithm for the case when the parameters are unknown. Section~\ref{sec:conclusion} concludes the paper. Proofs of the few technical results are collected in the Appendix.

\emph{Notations:} We use $\bN$, $\bZ_{\geq 0}$, $\bR$, $\bR_+$ to denote the set of natural numbers, non-negative integers, real numbers, and non-negative real numbers, respectively. The probability of an event is denoted by $\bP(\cdot)$. For a $\sigma$-algebra $\cF$, $\bP[\cdot \mid \cF]$ and $\bE[\cdot \mid \cF]$ denote the conditional probability and conditional expectation given $\cF$, respectively.~For a topological space $\cX$, $\bB(\cX)$ denotes its Borel $\sigma$-algebra on $\cX$. We write $\eta(x;\mu,\sigma^2)$ for the probability density function (pdf) of a Gaussian random variable with mean $\mu$ and variance $\sigma^2$, and $\mathrm{Geom}(p)$ for the geometric distribution with parameter $p$. For an event $A$, $\mathds{1}(A)$ denotes its indicator random variable. The Dirac measure
with unit mass at $x$ is denoted by $\delta_x(\cdot)$. For a matrix $M$, $M^T$ denotes its transpose. For a set $\cX$, $2^{\cX}$ denotes its power set. The symbol $\otimes$ denotes the product $\sigma$-algebra. In particular, if $(\cX_i,\cF_i)$, $i=1,\ldots,n$, are measurable spaces, then $\bigotimes_{i=1}^n \cF_i$ is the product $\sigma$-algebra on $\prod_{i=1}^n \cX_i$.
\section{Problem formulation} \label{sec:prob_form}

\subsection{System Model}

\begin{figure}
	\centering
	\includegraphics[width=0.9\linewidth, trim=0.6cm 0cm 0.55cm 0cm, clip]{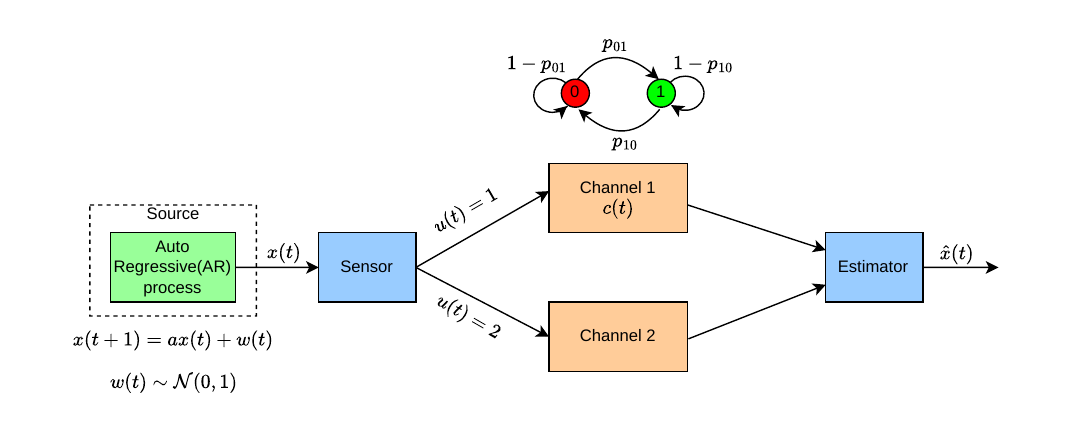}
	\caption{Remote state estimation setup with hybrid communication channels.}
	\label{fig:rsehybridv2}
\end{figure}

We consider the remote state estimation setup shown in Fig. 1, which consists of an autoregressive (AR) source, a sensor, and a remote estimator. The discrete-time AR process $\{x(t)\}_{t \in \bZ_{\geq 0}}$ evolves as follows:
\al{\label{source}
x(t+1) = a x(t) + w(t),~t = 0, 1, \ldots,
}
where $x(0) \sim \cN(0,1)$, $a, x(t) \in \bR$, and $\{w(t)\}_{t \in \bZ_{\geq 0}}$ is an i.i.d. Gaussian noise process with $w(t) \sim \cN(0,1)$. The sensor observes the source perfectly and and sends encoded observations to a remote estimator over one of two communication channels. At each time $t$, the sensor chooses either the fast but unreliable channel, denoted by $u(t
)=1$, or the slow but reliable channel, denoted by $u(t)=2$. The two channels differ in their reliability and delay characteristics~\cite{aziz2016architecture}. We call the fast but unreliable channel as ``Channel 1,'' and slow but reliable channel as ``Channel 2.'' In particular, Channel 1 channel has a delay of $1$ time unit, whereas Channel 2 delivers packets with probability $1$ but has a fixed delay of $d$ time units. The coexistence of these heterogeneous delays is a key feature of our model and plays an important role in the scheduling problem.

Channel 1 is modeled as a Gilbert-Elliott (GE) channel~\cite{gilbert1960capacity}. Let $c(t) \in \{0,1\}$ denote its state at time $t$. The state $c(t)=1$ represents a good channel state, in which a packet transmitted by the sensor at time $t$ is delivered successfully to
the estimator after one time step. The state $c(t)=0$ represents a bad channel state, in which a transmitted packet is dropped. Let $p_{01} \in (0,1)$ denote the probability that the channel transitions to state $1$ at the next time step given that it is currently in state $0$, and let $p_{10} \in (0,1)$ denote the probability that it transitions to state $0$ given that it is currently in state $1$. We assume that the sensor observes the channel state with a one-step delay, i.e., at time $t$, it knows $c(t-1)$.

Channel 2, on the other hand, delivers each packet to the estimator after a deterministic delay of $d$ time units. We let $r(t) \in \{0,1,\ldots, d-1\}$ denote the remaining time left for the packet currently in transit over Channel 2 to get delivered to the remote estimator. The value $r(t) = 0$ indicates that Channel 2 is available for a new transmission. Note that since Channel 1 has a delay of $1$ time unit, it is always available for transmission at each $t$. At each time $t$, the sensor knows $r(t)$. We assume that the two channels cannot be used simultaneously for transmission. Hence, if Channel 2 is occupied with a packet in transit, i.e., $r(t) \geq 1$, the sensor cannot initiate a new transmission. and we set $u(t) = 0$. Thus, $u(t) = 0$ corresponds to the situation where the sensor remains idle while the previously transmitted packet over Channel 2 is yet to be delivered to the estimator. We assume that each transmission over either channel incurs an energy cost of $\lambda > 0$. Moreover, we assume the non-trivial case of $d \ge 2$. This is because when $d = 1$, then both the channels have the same transmission delay and incur the same energy costs with Channel 2 being perfectly reliable. Hence, it is always optimal to use Channel 2.

Let $\hat{x}(t)$ denote the estimate at the remote estimator. It is updated recursively as follows: for $t \geq 0$ we have,
\al{\label{xhat_evolv}
\hat{x}(t+1) =
\begin{cases}
    a x(t) &\mbox{ if } u(t)c(t) = 1,\\
    a^{d} x(t-d+1) &\mbox{ if } r(t) = 1,\\
    a \hat{x}(t) &\mbox{ otherwise},
\end{cases}
}
where $\hat{x}(0) =0$. The second case corresponds to the event that a packet transmitted over Channel 2 at time $t-d+1$ is delivered to the estimator at time $t+1$. Thus, the term $x(t-d+1)$ is well-defined whenever this event occurs.~Let $\D(t) := x(t) - \hat{x}(t)$ denote the estimation error at time $t$. From~\eqref{source} and~\eqref{xhat_evolv}, the error evolves as
\al{\label{e_evolve}
\D(t+1) = 
\begin{cases}
    w(t) &\mbox{ if } u(t)c(t) = 1,\\
    \sum_{s= t-d+1}^{t} a^{t-s} w(s) &\mbox{ if } r(t) = 1,\\
    a\D(t) +w(t) &\mbox{ otherwise}.
\end{cases}
}

\subsection{Optimization Problem}

Let $\cI(t)$ denote the information available to the sensor at time $t$, i.e., $\cI(t) = (\{x(s),r(s)\}_{s=0}^{t}, \{c(s),u(s)\}_{s=0}^{t-1})$. We use $\pi = (\pi_0, \pi_1, \ldots)$ to denote a scheduling policy. We say a scheduling policy $\pi$ is admissible if: (i) $u(t) = \pi_t(\cI(t))$, i.e., where $\pi_t$ is a measurable function such that $\pi_t : \cI(t) \mapsto u(t) \in \{0,1,2\}$, and (ii) $u(t) = 0$ whenever $r(t) \geq 1$, i.e., no transmission occurs when Channel 2 is busy. 

We are interested in designing a scheduling policy $\pi$ that solves the following infinite horizon average cost problem: 
\al{
\min_{\pi} \limsup_{T \to \infty} \frac{1}{T} \bE_{\pi} \left[ \sum_{t=0}^{T-1} \left( \Delta(t)^2 + \lambda \mathds{1}(u(t) \in \{1,2\}) \right) \right], \label{opt_prob}
}
where $\bE_{\pi}$ denotes expectation with respect to the probability measure induced by an admissible policy $\pi$, and the $t$-th summand above denotes the instantaneous cost at time step $t$.

\section{MDP formulation and Value Iteration}

We first formulate the optimization problem~\eqref{opt_prob} as an MDP in Section~\ref{subsec:MDP formulation}. Next, we introduce the $\beta$-discounted cost problem~\cite{Hernandez2012discrete} in Section~\ref{subsec:discount MDP}, which is easier to analyze than the infinite horizon average cost problem. Since the state-space is uncountable and the instantaneous cost function is unbounded, it is not immediately clear whether the value iteration (VI) algorithm can be applied to solve the $\beta$-discounted problem. Nevertheless, we show that, under a mild assumption on the AR process parameter, $a$, and the parameters of Channel 1, the VI algorithm converges to an optimal solution of the $\beta$-discounted problem.

\subsection{MDP Formulation} \label{subsec:MDP formulation}
The optimization problem~\eqref{opt_prob} can be modeled equivalently as the following MDP.

\begin{enumerate}
    \item[i)] The \emph{state} of the MDP at time $t$ is
    \al{\label{state_MDP}
    s(t) = (\D(t),r(t),c(t-1)),
    }
    where $\D(t) \in \bR$ is the estimation error at time $t$, $r(t) \in \{0,1,\ldots,d-1\}$ is the remaining time left for the in transit over Channel 2 to get delivered at time $t$, and $c(t-1) \in \{0,1\}$ is the GE channel state of Channel 1 at time $t-1$. We denote the state-space by $\mathcal{S} := \bR \times \{0,1,\ldots,d-1\} \times \{0,1\}$.
    
    \item[ii)] The \emph{decision} taken by the sensor at time $t$ is denoted by $u(t) \in \mathcal{U}(s)$, where $\mathcal{U}(s)$ is the set of feasible actions that can be taken when the state of the MDP is $s$. Specifically, for $s = (\D, r, c)$,
    \[
    \cU(s) =
    \begin{cases}
        \{0\} & r \geq 1,\\
        \{1,2\} & r = 0.
    \end{cases}
    \]
    Let $\cU := \{0,1,2\}$ denote the overall action set.
    \item[iii)] The \emph{instantaneous cost} incurred in state $s=(\D,r,c)$ when action $u$ is applied is given by
    \al{
    \ell(s, u) := \D^2 + \lambda \lf(\mathds{1} (u=1) + \mathds{1} (u = 2)\rt). \label{instan_cost}
    }

    \item[iv)] Let $s_+ = (\D_+, r_+, c_+)$. We use $p(s_+ \mid s;u)$ to denote the \emph{transition density function} from the current state $s$ to the next state $s_+$ when decision $u$ is taken. These density functions for each $u \in \cU$ are illustrated in Table~\ref{tab:df} with $\sigma^2_d := \sum_{k=0}^{d-1} (a^2)^k$.
\end{enumerate}

\begin{table}[ht]
\caption{Transition density function\label{tab:df}}
    \renewcommand{\arraystretch}{1.5}
    \centering
    \begin{tabular}{|p{1em}|p{1.9cm}|p{17em}| } 
         \hline
         $u$ & $s$ & $p(s_+ \mid s;u)$\\
         \hline
         \multirow{3}{9em}{0} & $(\D,r,c)$, $r \geq 2$ & $p_{cc_+} \eta(\D_+;a\D, 1) \delta_{r-1}(r_+) \delta_{c_+}(c_+)$\\
         & $(\D,1,c)$ & $p_{cc_+} \eta(\D_+;0, \sigma^2_d) \delta_{0}(r_+) \delta_{c_+}(c_+)$ \\
         \hline
         \multirow{2}{6em}{1} & \multirow{2}{4em}{$(\D,0,c)$} & $p_{c0} \eta(\D_+;a\D, 1) \delta_{0}(r_+) \delta_{0}(c_+)$ \\
         & & $p_{c1} \eta(\D_+;0, 1) \delta_{0}(r_+) \delta_{1}(c_+)$\\
         \hline
         2 & $(\D,0,c)$ & $p_{cc_+} \eta(\D_+;a\D, 1) \delta_{d-1}(r_+) \delta_{c_+}(c_+)$\\
         \hline
    \end{tabular}
    
\end{table}

\subsection{Discounted MDP}\label{subsec:discount MDP}

To solve~\eqref{opt_prob}, we begin by analyzing the corresponding $\beta$-discounted cost problem as described next. In Section~\ref{sec:struc_res}, we establish key structural properties of an optimal policy for the $\beta$-discounted cost problem. We then extend these results to the average cost problem~\eqref{opt_prob}.

We consider the following $\beta$-discounted MDP:

\al{
\min_{\pi} \bE_{\pi} & \left[\sum_{t = 0}^{\infty} \beta^t \left(\D(t)^2 + \lambda \mathds{1}(u(t) \in \{1,2\}) \right)\right], \label{discount_opt_prob}
}
where $\beta \in (0,1)$ is the discount factor and $\pi$ is an admissible policy corresponding to the above $\beta$-discounted optimization problem. The optimization problem in \eqref{discount_opt_prob} is defined for an initial state $s(0)= (\D(0), r(0), c(-1)) \in \cS$. The initial estimation error $\D(0)$ satisfies $\D(0) \sim \cN(0,1)$ since the initial state satisfies $x(0) \sim \cN(0,1)$ and the initial estimate is taken to be $\hat{x}(0)=0$. Moreover, $r(0) \in  \{0,1,\ldots ,d-1\}$, and $c(-1) \in \{0,1\}$ represents the initial Channel 2 delay and Channel 1 state, respectively.

\subsection{Value Iteration} \label{subsec:VI}
We now show that, under the following assumption on the AR process parameter and the parameters of Channel 1, we can use the VI algorithm to solve the $\beta$-discounted MDP~\eqref{discount_opt_prob}.
\begin{assumption}\label{ass:sys_chn}
    The AR process parameter $a$ and Channel 1 transition probabilities satisfy
    \nal{
    a^2 (1 - p_{01}) < 1.
    }
\end{assumption}

\begin{remark}
    Assumption~\ref{ass:sys_chn} is a mean-square stability condition as shown in~\cite[Theorem 8]{huang2007stability}. It ensures that the expected squared estimation error remains bounded and prevents the error variance from diverging during consecutive Channel 1 transmission failures. Similar assumption is also considered in several works, see~\cite{ chakravorty2019remote, dutta2023optimal,wu2017kalman}.
\end{remark}

To use value iteration, we first verify that the discounted MDP~\eqref{opt_prob} satisfies the required conditions~\cite[p.~46]{Hernandez2012discrete}. For this purpose, we define the infinite horizon expected total $\beta$-discounted cost given an initial state $s$, and instantaneous cost $\ell$~(cf.~\eqref{instan_cost}) as follows:
\al{
J\ub(s;\pi) := \bE_{\pi} \lf[\sum_{t = 0}^{\infty} \beta^t \ell(s(t),u(t))\rt]. \label{j_beta}
}
\begin{lemma} \label{lemma:VI_assump}
    Under Assumption~\ref{ass:sys_chn}, the $\beta$-discounted cost MDP (cf.~\eqref{discount_opt_prob}) satisfies the following properties:
    \begin{itemize}
        \item[P1)]  The instantaneous cost function $\ell(s,u)$ in \eqref{instan_cost} is continuous, non-negative, and inf-compact on $\mathcal{K}$, where $\mathcal{K} := \{(s,u) \mid s \in \mathcal{S}, u \in \mathcal{U}(s)\} \subseteq \mathcal{S} \times \cU$ is the set of feasible state-action pairs that is a measurable subset of $\mathcal{S} \times \cU$.

        \item[P2)] The transition kernel $\cP(\cdot \mid s,u)_{(s,u) \in \mathcal{K}}$ that describes the transition probabilities from the current state $s$ to the next state under the application of action $u$, is strongly continuous in $(s,u)$.

        \item[P3)] There exists a policy $\pi$ such that $J\ub(s,\pi) < \infty$~(cf.~\eqref{j_beta}), for every $s \in \cS$.
    \end{itemize}
\end{lemma}

\begin{proof}
	P1) follows from the definition of $\ell$ in~\eqref{instan_cost} and since the action set $\cU$ is finite.
	
	P2) Let the state-space $\mathcal S$ be endowed with the product $\sigma$-algebra $\bB(\bR) \otimes 2^{\{0,\ldots,d-1\}} \otimes 2^{\{0,1\}}$. From Table~\ref{tab:df}, for each $(s,u)$, the transition kernel admits a density $p(s_+ \mid s;u)$ with respect to the product measure consisting of the Lebesgue measure on $\bR$ and the counting measure on $\{0,\ldots,d-1\} \times \{0,1\}$, i.e., for any $\cB \in \bB(\bR)$, 
	\nal{
	 & \cP((\D_+, r_+, c_+) \in (\cB \times \{0,1,\ldots,d-1\} \times \{0,1\}) \mid s,u) \\
	& = \sum_{r_+,c_+} \int_{\cB} p(s_+ \mid s,u) \, d\Delta_+.
	}
	
	Now, we fix $s_+ = (\Delta_+, r_+, c_+) \in \mathcal S$. For each $u \in \mathcal U$, the density $p(s_+ \mid s;u)$ (cf. Table~\ref{tab:df}) is given by Gaussian densities in $\Delta_+$ whose mean and variance depend continuously on $\Delta$, along with Dirac terms corresponding to the discrete components $(r_+, c_+)$.
	
	For any fixed $(r,c)$, the mapping $\Delta \mapsto p(s_+ \mid s;u)$ is continuous. Since the discrete components take values in finite sets, it follows that $(s,u) \mapsto p(s_+ \mid s;u)$ is continuous for every fixed $s_+$. 
	Hence, by~\cite[Example C.6]{Hernandez2012discrete}, the transition kernel $\cP(\cdot \mid s,u)$ is strongly continuous.
	
    P3) We divide the proof into the following two cases. For an initial state with $r(0) = r$, let $\pi^{(01)}$ denote the policy that applies $u = 0$ whenever $r \ge 1$, and $u = 1$ whenever $r = 0$.
    
    Case a): $r(0) = 0$, i.e., the system starts with an initial state $(\D, 0, c)$. 
    Then we can write $\D(t)$ as follows: 
    \nal{
    \D(t) = \lf(\prod_{m = 0}^{t-1} a(m)\rt) \D + \sum_{k = 0}^{t-1} \lf(\prod_{m = k+1}^{t-1} a(m)\rt) w(k),
    }
    with the understanding that if $m > t-1$, then $\lf(\Pi_{m = k + 1}^{t-1} a(m)\rt) = 1$, and where
    \al{\label{eq:a(m)}
    a(m) = 
    \begin{cases}
        a &\mbox{ if } c(m) = 0,\\
        0 &\mbox{ if } c(m) = 1.
    \end{cases}
    }
    Now, since $\{w(k)\}_{k \in \bZ_{\geq 0}}$ are i.i.d., ${\cal N}(0,1)$-distributed, and independent of $\{a(k)\}_{k \in \bZ_{\geq 0}}$, we have that
    \al{
    & \bE\lf[\beta^t \D(t)^2\rt] \notag \\
    =~& \beta^t \D^2 \bE\lf[\lf(\prod_{m = 0}^{t-1} a(m)\rt)^2\rt]   + \beta^t  \bE\lf[\sum_{k = 0}^{t-1} \lf(\prod_{m = k+1}^{t-1} a(m)\rt)^2\rt] \notag \\
    =~& \beta^t \D^2 \bE\lf[\bE\lf[\lf(\prod_{m = 0}^{t-1} a(m)\rt)^2 \Big| c(0)\rt]\rt] \notag \\
    & \qquad + \beta^t \bE\lf[\sum_{k = 0}^{t-1} \bE\lf[\lf(\prod_{m = k+1}^{t-1} a(m)\rt)^2 \Big| c(k)\rt]\rt] \notag \\
    =~& \beta^t \D^2 \bP(c(0) = 0) a^{2t} (1- p_{01})^{t-1} \notag \\
    & \quad + \beta^t  \sum_{k=0}^{t-1} \bP(c(k) = 0) a^{2(t-k-1)} (1-p_{01})^{t-k-2} \label{eq:for_lower_bd} \\
    \leq~& \beta^t \D^2 \bP(c(0) = 0) a^{2t} (1- p_{01})^{t-1} \notag \\
    & \quad + \beta^t \frac{1}{(1-p_{01}) (1 - a^2(1-p_{01}))}, \notag
    }    
    where the second equality follows from law of total expectation. The third equality follows since $\prod_{m=k}^{t-1} a(m)$ is either equal to 0 if $c(k) = 1$, or is equal to $a^{2(t-k)}$ if $c(m) = 0$ for all $m \in \{k, k+1, \ldots, t-1\}$. The last inequality follows since $\bP(c(k) = 0) \leq 1$ and $a^2(1-p_{01}) < 1$ (Assumption~\ref{ass:sys_chn}). Now since the cost per transmission is $\lambda$, summing them over $t$ we have the following bound for an initial state $s$ which follows from~\eqref{j_beta}:
    \al{
    & J\ub(s;\pi^{(01)}) \leq \frac{\D^2}{(1-p_{01}) (1 - \beta a^2(1-p_{01}))} \notag \\
    & \qquad + \frac{1}{1-\beta}\lf[\frac{1}{(1-p_{01}) (1 - a^2(1-p_{01}))} + \lambda \rt] \label{eq:upper_bound}.
    }

    Case b): $r(0) = r \geq 1$, i.e., the system starts with an initial state $(\D,r,c)$. Then, $\D(t)$ for $t < r $ satisfies the following
    \al{
    \D(t) = a^t \D(0) + \sum_{k = 0}^{t-1} a^{t-1-k} w(k)
    }
    where $\D(0) \sim \cN(0,1)$, and for $t > r$ we have that,
    \nal{
    \D(t) = \lf(\prod_{m = r}^{t-1} a(m)\rt) \D(r) + \sum_{k = r}^{t-1} \lf(\prod_{m = k+1}^{t-1} a(m)\rt) w(k),
    }
    where $a(m)$ is as defined in~\eqref{eq:a(m)}, and
    \al{
    \D(r) = \sum_{t=0}^{r-1} a^{r-1-t} w(t). \label{eq:d(r)}
    }
    Next, it follows from~\eqref{eq:d(r)} that $\bE[\D(r)^2] = (1-(a^2)^r)/(1-a^2)$. Then, for $t < r$ we have that,
    \al{
    \bE\lf[\beta^t \D(t)^2\rt] &= \beta^t a^{2t} \bE\lf[\D(0)^2\rt] + \sum_{k = 0}^{t-1} \beta^k a^{2(t-1-k)} \bE\lf[w(k)^2\rt] \notag \\
    & = (\beta a^2)^t + \frac{1 - (\beta a^2)^t}{1- \beta a^2} \notag \\
    & = \frac{1 - (\beta a^2)^{t+1}}{1 - \beta a^2}, \label{eq:b1}
    }
    and for $t > r$, similar to Case a) we have,
    \al{
    &\bE\lf[\beta^t \D(t)^2\rt] = \beta^t \bE\lf[\D(r)^2\rt] \bE\lf[\lf(\prod_{m = r}^{t-1} a(m)\rt)^2\rt] \notag \\
    & + \beta^t \bE\lf[\sum_{k = 0}^{t-1} \lf(\prod_{m = k+1}^{t-1} a(m)\rt)^2\rt] \notag \\
    & \leq \beta^t \frac{1 + (1-(a^2)^r)/(1-a^2)}{(1-p_{01}) (1 - a^2(1-p_{01}))}. \label{eq:b2}
    }
    Thus, combining~\eqref{eq:b1}, ~\eqref{eq:b2} with~\eqref{j_beta}, we have the following bound for initial state $s$:
    \nal{
    & J\ub(s;\pi^{(01)}) = \sum_{t = 0}^{r-1} \bE[\beta^t \D(t)^2] \notag \\
    & + \beta^r \bE[\D(r)^2] + \sum_{t=r+1}^{\infty} \bE[\beta^t \D(t)^2] + \frac{\lambda}{1-\beta}\notag \\
    & \leq \frac{1}{1 - \beta a^2} \lf[r - \beta a^2 \lf(\frac{1- (\beta a^2)^r}{1 - \beta a^2}\rt)\rt] + \frac{1 - (a^2)^r}{1- a^2}\\
    & + \frac{1}{1 - \beta}\lf[\frac{(1-(a^2)^r)/(1-a^2) + 1}{(1-p_{01}) (1 - a^2(1-p_{01}))} + \lambda \rt]. 
    }
    This completes the proof.
\end{proof}

Lemma~\ref{lemma:VI_assump} allows us to apply the VI algorithm for solving the MDP~\eqref{discount_opt_prob}, which we will now describe. For this, we define the VI functions $\{V\ub_n(\cdot)\}_{n \in \bZ_{\geq 0}}$ as follows: for $s =(\D,r,c) \in \cS$, 
\al{
V_0(s) = 0, \label{eq:V_0}
}
and,
\al{
V\ub_n(s) = \min_{u \in \cU(s)} Q\ub_n(s;u), \label{eq:V_n}
}
where 
\al{
Q\ub_n(s;u) = \ell(s,u) + \beta \bE_{s_+ \sim p(\cdot \mid s; u)}\lf[V\ub_{n-1}(s_+)\rt], \label{eq:Q_n}
}
where $s_+ \sim p(\cdot \mid s; u)$ indicates that the next state $s_+$ is sampled from the probability distribution $p(\cdot \mid s; u)$.

We use $V\ub(s)$ to denote the value function for an initial state $s$, i.e.,
\al{
V\ub(s) = \min_{\pi} J\ub(s;\pi), \label{def:V}
}
where $J\ub(s;\pi)$ is defined in~\eqref{j_beta}. 

The following proposition shows the convergence of the VI algorithm to $V\ub$, and establishes the existence of an optimal deterministic stationary policy for the MDP~\eqref{discount_opt_prob}. The proof of this proposition follows from Lemma 4.2.8 and Theorem 4.2.3 of~\cite{Hernandez2012discrete} which uses Lemma~\ref{lemma:VI_assump}.

\begin{proposition} \label{prop:VI}
Suppose Assumption~\ref{ass:sys_chn} holds. Then for the MDP~\eqref{discount_opt_prob} we have,
    \begin{enumerate} 
        \item[a)] The VI algorithm~\eqref{eq:V_n}-\eqref{eq:Q_n} converges to $V\ub$, i.e., $\lim_{n \to \infty} V\ub_n(s) = V\ub(s)$ for each $s \in \cS$.

        \item[b)] The value function $V\ub$ in ~\eqref{def:V} satisfies the following for each $s \in \cS$:
        \al
        {
            V\ub(s) = \min_{u \in \cU(s)} Q\ub(s;u), \label{eq:V}
        }
        where 
        \al
        {
            & Q\ub(s;u)  = \ell(s,u) + \beta \bE_{s_+ \sim p(\cdot \mid s; u)}\lf[V\ub(s_+)\rt]. \label{eq:Q}
        }

        \item[c)] There exists a deterministic policy that attains the minimum in~\eqref{eq:V} for each $s \in \cS$. Such a policy solves~\eqref{def:V} for each $s \in \cS$ and is therefore optimal for the MDP~\eqref{discount_opt_prob}.
    \end{enumerate}
\end{proposition}

\section{Structural results} \label{sec:struc_res}

In this section, we first establish structural properties of an optimal policy for the $\beta$-discounted problem~\eqref{discount_opt_prob}. We then show in Section~\ref{subsec:struct_avg} that these structural properties extend to the average cost problem~\eqref{opt_prob}. 

We begin by defining threshold-type policies for~\eqref{opt_prob} and~\eqref{discount_opt_prob}. Recall that the sensor can attempt
transmission over Channel 1, $u(t)=1$, or Channel 2, $u(t)=2$, only when $r(t)=0$. This motivates the following policy class.

\begin{definition}
	\textbf{(i) Non-decreasing threshold-type policy in $\Delta$:} 
	A scheduling policy $\pi : \mathcal{S} \to \{1,2\}$ is said to be non-decreasing threshold-type in $\Delta$ if, for each $c\in\{0,1\}$, there exists a threshold $\Delta^\ast(c)\geq 0$ such that, for $(\D,r,c)\in\cS$,
	\[
	\pi(\Delta,r,c) =
	\begin{cases}
        0 & \text{if } r \ge 1, \\
		1 & \text{if } r=0 \text{ and } |\Delta| \leq \Delta^\ast(c), \\
		2 & \text{if } r=0 \text{ and } |\Delta| > \Delta^\ast(c).
	\end{cases}
	\]
	
	\textbf{(ii) Non-increasing threshold-type policy in $\Delta$:} 
	A scheduling policy $\pi : \mathcal{S} \to \{1,2\}$ is said to be non-increasing threshold-type in $\Delta$ if, for each $c\in\{0,1\}$, there exists a threshold $\Delta^\ast(c)\geq 0$ such that, for $(\D,r,c)\in\cS$,
	\[
	\pi(\Delta,r,c) =
	\begin{cases}
        0 & \text{if } r \ge 1, \\
		2 & \text{if } r=0 \text{ and } |\Delta| \leq \Delta^\ast(c), \\
		1 & \text{if } r=0 \text{ and } |\Delta| > \Delta^\ast(c).
	\end{cases}
	\]
\end{definition}

\subsection{Analysis of the Discounted MDP} \label{subsec:struct_dis}

In this section, we show that the discounted MDP~\eqref{discount_opt_prob} admits an optimal threshold-type policy. We begin by defining the following four distinct regions to analyze the structural properties of an optimal policy for the MDP~\eqref{discount_opt_prob}. For ease of notation, we use $\gamma_1 := \beta a^2 (1-p_{01}), \gamma_2:= \sum_{i=0}^{d-1} \lf(\beta a^2\rt)^i$.

\begin{definition} \label{def:4reg_dis}
    Consider the $\beta$-discounted MDP with discount factor $\beta \in (0,1)$ (cf.~\eqref{discount_opt_prob}), and suppose that Assumption~\ref{ass:sys_chn} holds. 
    We consider the parameter space $(a, d, p_{01}, p_{10}) \in  \bR \times \{2,3,\ldots\} \times (0,1)^2$, where $a \in \bR$ is the AR process parameter~(cf.~\eqref{source}), $d \in \{2,3,\ldots\}$ is Channel 2 delay, and $(p_{01}, p_{10}) \in (0,1)^2$ are the parameters of Channel 1. We partition the admissible parameter space $\cD\ub := \{(a, d, p_{01}, p_{10}): \g_1 < 1\}$ into four regions. We note that the condition $\g_1 < 1$ follows from Assumption~\ref{ass:sys_chn}. The regions are defined as follows:
    \nal{
    R\ub_1 & := \{(a,d,p_{01},p_{10}) : \\
    &\quad F^{(\beta)}(a,d,p_{01},p_{10}) \le 0,\;
        H^{(\beta)}(a,d,p_{01},p_{10}) \le 0 \}, \\
    R\ub_2 & := \{(a, d,p_{01}, p_{10}): \\
    &\quad F^{(\beta)}(a, d,p_{01}, p_{10}) \le 0, H^{(\beta)}(a, d,p_{01}, p_{10}) > 0 \}, \\
    R\ub_3 & := \{(a, d,p_{01}, p_{10}): \\
    & \quad F^{(\beta)}(a, d,p_{01}, p_{10}) > 0, G^{(\beta)}(a, d, p_{01}, p_{10}) \leq 0 \}, \\
    R\ub_4 & := \{(a, d,p_{01}, p_{10}): \\
    & \quad F^{(\beta)}(a, d,p_{01}, p_{10}) > 0, G^{(\beta)}(a, d,p_{01}, p_{10}) > 0 \},
    }
    where the functions $F\ub(\cdot)$, $H\ub(\cdot)$, $G\ub(\cdot) : \bR \times \{2,3,\ldots\} \times (0,1)^2 \mapsto \bR$ are defined as follows:
    \nal{
    & F^{(\beta)}(a, d,p_{01}, p_{10}) := \sum_{i=0}^{\infty} \g_1^i - \g_2, \\
    & H^{(\beta)}(a, d,p_{01}, p_{10}) := 1 + \beta a^2 p_{10} \sum_{i=0}^{\infty} \g_1^i - \g_2, \\
    & G^{(\beta)}(a, d,p_{01}, p_{10}) := 1 + \beta a^2 p_{10} \g_2 - \g_2.
    }
\end{definition}
The regions $R_1\ub, R_2\ub, R_3\ub, R_4\ub$ are illustrated in Fig.~\ref{fig:regions} for $\beta = 0.99$.

\begin{figure}[htbp]
	\begin{centering}
		\includegraphics[scale=0.45]{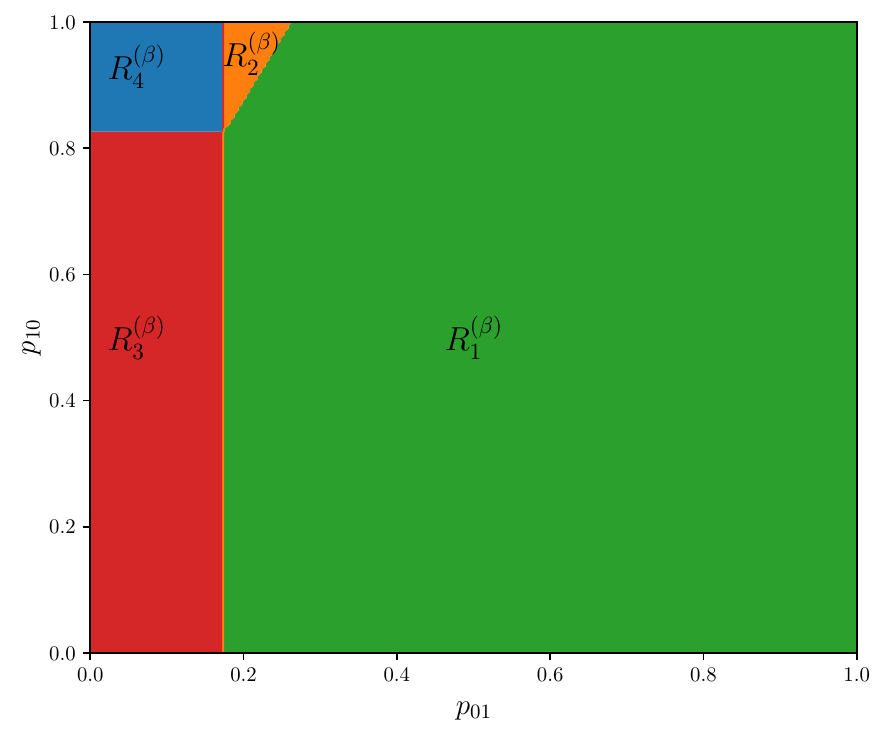} 
		\par\end{centering}
	\caption{Partition of the parameter space into four regions $R_1^{(\beta)}, R_2^{(\beta)}, R_3^{(\beta)}$, and $R_4^{(\beta)}$ shown on the $(p_{01},p_{10}) \in (0,1)^2$ plane for parameters $d=4$, $a = 0.9$ and $\beta =0.99$. The figure illustrates how the conditions on $F^{(\beta)}(\cdot)$, $H^{(\beta)}(\cdot)$, and $G^{(\beta)}(\cdot)$ partition the $(p_{01},p_{10})$ space.}
	\label{fig:regions}
\end{figure}

The following theorem is the main result of this section. It shows the existence of an optimal policy that is of threshold-type for the $\beta$-discounted problem~\eqref{discount_opt_prob}. 

\begin{theorem} \label{thm:discount}
    Consider the MDP~\eqref{discount_opt_prob} and let Assumption~\ref{ass:sys_chn} hold. Then, for each $\beta \in (0,1)$, there exists an optimal policy, denoted by $\pi\ub(\cdot)$, that satisfies the following: 
    \begin{enumerate}
        \item[1)] if $(a, d, p_{01}, p_{10}) \in R\ub_1$, then $\pi\ub(\cdot,0,c)$ is a non-increasing threshold-type policy in $\D$ for each $c \in \{0,1\}$;
        \item[2)] if $(a, d, p_{01}, p_{10}) \in R\ub_2$, then $\pi\ub(\cdot,0,0)$ is a non-increasing threshold-type policy and $\pi\ub(\cdot,0,1)$ is a non-decreasing threshold-type policy in $\D$;
        \item[3)] if $(a, d, p_{01}, p_{10}) \in R\ub_3$, then $\pi\ub(\cdot,0,0)$ is a non-decreasing threshold-type policy and $\pi\ub(\cdot,0,1)$ is a non-increasing threshold-type policy in $\D$;
        \item[4)] if $(a, d, p_{01}, p_{10}) \in R\ub_4$, then $\pi\ub(\cdot,0,c)$ is a non-decreasing threshold-type policy in $\D$ for each $c \in \{0,1\}$.
    \end{enumerate}
\end{theorem}

\begin{proof}
    We first note that proving the existence of a non-decreasing threshold-type policy in $\D$ for $c \in \{0,1\}$ is equivalent to showing that for any $\D_1, \D_2 \in \bR$ such that $|\D_1| \leq |\D_2|$, we have $Q\ub(\D_2,0,c;2) - Q\ub(\D_1,0,c;2) \leq Q\ub(\D_2,0,c;1) - Q\ub(\D_1,0,c;1)$. Similarly, the existence of non-increasing threshold-type policy in $\D$  is equivalent to showing that $Q\ub(\D_2,0,c;2) - Q\ub(\D_1,0,c;2) \geq Q\ub(\D_2,0,c;1) - Q\ub(\D_1,0,c;1)$. We will prove the result for $\D \in \bR_+$. The proof will then follow from Lemma~\ref{lemma:even} for $\D \in \bR$ since $\pi\ub(\cdot, 0, c)$ are even for each $c \in \{0,1\}$. We begin by observing that for each $\D_2, \D_1 \in \bR$, 
    \al{
    Q\ub(\D_2,0,c;2) - Q\ub(\D_1,0,c;2) = \g_2(\D_2^2 - \D_1^2), \label{eq:Q2_diff}
    } 
    where the equality follows from~\eqref{eq:lm2_2} in Lemma~\ref{lemma:V_r2_lb} and Proposition~\ref{prop:VI} since $\lim_{n \rightarrow \infty} Q\ub_n(s) = Q\ub (s)$ for each $s \in \cS$. We now start proving the result. 

    1) We will first show that there exists a constant $M \in \bR_+$ such that for each $\D \in \bR_+$, 
    \al{
    V\ub(\D,0,0) \leq \sum_{i=0}^{\infty} \g_1^i \D^2 + M. \label{eq:thm1_1}
    } 
    To show~\eqref{eq:thm1_1} we will use induction on VI functions $\{V_n\ub\}_{n \in \bN}$. The inequality will then follow from Proposition~\ref{prop:VI} since $\lim_{n \rightarrow \infty} V_n\ub(s) = V\ub(s)$ for each $s \in \cS$. Now, it follows from~\eqref{eq:V_0} and~\eqref{eq:V_n} that $V\ub_1(\D,0,0) = \D^2 \leq \sum_{i=0}^{\infty} \g_1^i \D^2$. This is the base case with $M_1 = 0$. Now assume that $V\ub_k(\D,0,0) \leq \sum_{i=0}^{\infty} \g_1^i \D^2 + M_k$ for $k = 1,2, \ldots, n$ and for some constant $M_n$ independent of $\D^2$. We will show that this inequality also holds for $V\ub_{n+1}(\D,0,0)$. It follows from~\eqref{eq:V_n} and~\eqref{eq:Q_n} that
    \nal{
    	& V\ub_{n+1}(\D,0,0) \leq Q\ub_{n+1}(\D,0,0;1) \notag \\
    	& \leq \D^2 + \lambda+ \beta (1-p_{01}) \notag \\
    	& \times \int_{\bR} \eta(\D_+;a\D,1) \lf(\sum_{i=0}^{\infty} \g_1^i \D_+^2 + M_n\rt) \,d\D_+ \notag \\
    	& + \beta p_{01} \int_{\bR} \eta(\D_+;0,1) \lf(\sum_{i=0}^{\infty} \g_1^i \D_+^2 + M_n\rt) \,d\D_+ \notag \\
    	& = \lf(1 + \g_1\sum_{i=0}^{\infty} \g_1^i\rt) \D^2 + \beta \lf(\sum_{i=0}^{\infty} \g_1^i + M_n\rt) + \lambda \notag \\
    	& = \lf(\sum_{i=0}^{\infty} \g_1^i\rt) \D^2 + M_{n+1}^{(1)}, 
    }
    where $M_{n+1}^{(1)} = \beta \lf(\sum_{i=0}^{\infty} \g_1^i + M_n\rt) + \lambda$. The inequality follows from the hypothesis on $V\ub_n$. This completes the induction by taking $M = \lim_{n \rightarrow \infty} M_{n}^{(1)}$.
    
    Next, we divide the proof into two cases as follows:

    Case i) $c = 0$: Combining~\eqref{eq:thm1_1} with~\eqref{eq:Q} we have
    \nal{
    & Q\ub(\D_2,0,0;1) - Q\ub(\D_1,0,0;1) \notag \\
    & = (\D_2^2 - \D_1^2) + \beta (1-p_{01}) \int_{\bR} (\eta(\D_+;a\D_2,1)  \notag \\
    & \hspace{2cm}- \eta(\D_+;a\D_1,1))  V\ub(\D_+,0,0) \,d\D_+ \notag \\
    & \leq \sum_{i=0}^{\infty} \g_1^i (\D_2^2 - \D_1^2) \leq \g_2 (\D_2^2 - \D_1^2),
    }
    where the last inequality follows since $(a,d,p_{01}, p_{10}) \in R\ub_1$. Now, the above result combined with~\eqref{eq:Q2_diff} completes the proof for $c =0$.

    Case ii) $c=1$: We have from~\eqref{eq:Q} and~\eqref{eq:thm1_1} that
    \nal{
    & Q\ub(\D_2,0,1;1) - Q\ub(\D_1,0,1;1) \notag \\
    & = (\D_2^2 - \D_1^2) + \beta p_{10} \int_{\bR} (\eta(\D_+;a\D_2,1) \notag \\
    & \hspace{2cm} - \eta(\D_+;a\D_1,1)) V\ub(\D_+,0,0) \,d\D_+ \notag \\
    & \leq \lf(1 + \beta p_{10} a^2 \sum_{i=0}^{\infty} \g_1^i\rt) (\D_2^2 - \D_1^2) \leq \g_2 (\D_2^2 - \D_1^2),
    }
    where the last inequality follows since $(a,d,p_{01}, p_{10}) \in R\ub_1$. This completes the proof.

    2) The proof for $c=0$ follows on similar lines as Case i) of part 1) and hence is omitted. We now focus on the case when $c=1$. It follows from~\eqref{eq:Q} that 
    \nal{
    & Q\ub(\D_2,0,1;1) - Q\ub(\D_1,0,1;1) = (\D_2^2 - \D_1^2) \\
    & + \beta p_{10} \int_{\bR} (\eta(\D_+;a\D_2,1)  \notag \\
    & \hspace{2cm}- \eta(\D_+;a\D_1,1))  V\ub(\D_+,0,0) \,d\D_+ \notag \\
    & \geq \lf(1 + \beta p_{10} a^2 \sum_{i=0}^{\infty} \g_1^i\rt) (\D_2^2 - \D_1^2) \\
    & \geq \g_2(\D_2^2 - \D_1^2),
    }
    where the first inequality follows from Lemma~\ref{lemma:V_r2_lb} and the second inequality follows since $(a,d,p_{01}, p_{10}) \in R\ub_1$. This completes the proof for part 2).
    
    3) Case i) $c = 0$: 
    It follows from Lemmas~\ref{lemma:aux_res} and~\ref{lemma:V_r2_lb} that $Q\ub(\D_2,0,0;1) - Q\ub(\D_1,0,0;1) \geq (1 + \g_1\g_2)(\D_2^2 - \D_1^2) \geq \g_2 (\D_2^2 - \D_1^2)$.

    Case ii) $c=1$: It follows from~\eqref{eq:V} and~\eqref{eq:Q} that
    \al{
    & V\ub(\D,0,0) \le Q\ub(\D,0,0;2) \notag \\
    & = \D^2 + \lambda + \beta \sum_{c_1 \in \{0,1\}} p_{0c_1} \notag \\
    & \times \int_{\bR} \eta(\D_1';a\D,1) V\ub(\D_1', d-1, c_+) \,d\D_1'. \label{eq:part3_1}
    }
    We first focus on the term $V\ub(\D_1', d-1, c_+)$. 
    Because the optimal decision is $u = 0$ for $r \geq 1$, we can iteratively expand $V\ub(\D_1', d-1, c_1)$ as follows: 
    \nal{
    & V\ub(\D_1', d-1, c_1) \notag \\
    =~& Q\ub(\D_1', d-1, c_1;0) =  (\D_1')^2\notag \\
    +~& \beta \sum_{c_2} p_{c_1c_2} \int_{\bR} \eta(\D_2'; a\D_1', 1) V\ub(\D_2', d-2, c_2) \,d\D_2' \notag \\
    \vdots \notag \\
    =~& (\D_1')^2 + \beta (1 + (a\D_1')^2) + \ldots \notag \\
    +~& \beta^{d-2} (1 + a^2 + \ldots + (a^2)^{d-3} + (a^2)^{d-2} (\D_1')^2) \notag \\
    +~& \beta^{d-1} \sum_{c_d} p_{c_1c_d}^{(d-1)} \int_{\bR} \eta(\D_d'; 0, \sigma^2_d) V\ub(\D_d', 0, c_d) \,d\D_d', 
    }
    where $p_{cc_+}^{(n)}$ denotes the $n$-step transition probability from state $c$ to $c_+$ of Channel 1. Plugging the above in~\eqref{eq:part3_1}, we get
    \nal{
    & V\ub(\D,0,0) \leq Q\ub(\D,0,0;2) = \sum_{i=0}^{d-1} \lf(\beta a^2\rt)^i \D^2 + \Hat{K}^{(2)},
    }
    where $\Hat{K}^{(2)}$ is the constant term independent of $\D^2$. Now, using the above result we get $Q\ub(\D_2,0,1;1) - Q\ub(\D_1,0,1;1) \le (1 + \beta a^2 p_{10} \g_2)(\D_2^2 - \D_1^2) \le \g_2 (\D_2^2 - \D_1^2)$. This completes the proof of part 3).

    4) The proof for $c = 0$ follows along similar lines as part 3) and is therefore omitted. We now focus on the case when $c = 1$. It follows from Lemma~\ref{lemma:V_r2_lb} that $Q\ub(\D_2,0,1;1) - Q\ub(\D_1,0,1;1) \ge (1 + \beta a^2 p_{10} \g_2) (\D_2^2 - \D_1^2) \ge \g_2(\D_2^2 - \D_1^2)$. This completes the proof.
\end{proof}

\subsection{Intuition Behind the Regions $\{R\ub_i\}_{i=1}^4$} \label{subsec:intuition}

This section gives intuition for the regions $\{R\ub_i\}_{i=1}^4$ and interpret the results of Theorem~\ref{thm:discount}. The regions $\{R\ub_i\}_{i=1}^4$ capture the trade-off in the hybrid channels, i.e., the choice between the fast but unreliable Channel 1 and slow but reliable Channel 2. The decision taken by the sensor not only depends on the channel parameters $(d,p_{01},p_{10})$ but also on the AR parameter $a$, which governs how quickly the estimation error grows while the estimator waits for a delayed Channel 2 packet.

\begin{itemize}
    \item[1)] Region $R\ub_1$: Channel 1 has large $p_{01}$ and is therefore likely to recover from a bad state. This implies that Channel 1 is reliable. An optimal policy exploits its low latency and uses Channel 1 when the estimation error is large. Even if the previous Channel 1 state was bad ($c(t-1)=0$), a large $p_{01}$ makes it preferable to risk Channel 1 rather than wait $d$ time steps for Channel 2.

    \item[2)] Region $R\ub_2$: Channel 1 switches frequently between good and bad states because both $p_{01}$ and $p_{10}$ are large. An optimal policy anticipates this oscillation.  If Channel 1 state was bad in the previous time step ($c(t-1)=0$), it is likely to be in a good state in the next time step. Hence, an optimal policy chooses to transmit via Channel 1 for large estimation errors. However, if Channel 1 state was good in the previous time step ($c(t-1)=1$), it is likely to be in a bad state in the next time step. Hence, in this case, an optimal policy switches to Channel 2 when the magnitude of the estimation error is large.

    \item[3)] Region $R\ub_3$: This region indicates that if Channel 1 state is good, it tends to stay in the good state (low $p_{10}$), and if it is in a bad state, it tends to stay in the bad state (low $p_{01}$). An optimal policy therefore trust the previous state of Channel 1. If Channel 1 state was bad in the previous time step ($c(t-1)=0$), an optimal policy avoids Channel 1 for large errors and offloads to Channel 2. Conversely, if Channel 1 state was good in the previous time step ($c(t-1)=1$), an optimal policy uses Channel 1 for large errors.

    \item[4)] Region $R\ub_4$: Channel 1 is highly unreliable because $p_{01}$ is small and $p_{10}$ is large. An optimal policy uses Channel 1 only when the estimation error is small, and once the error exceeds a threshold, it is optimal to use the reliable Channel 2.

\end{itemize}

\subsection{Analysis of the Average Cost MDP} \label{subsec:struct_avg}

We next show that the structural properties of an optimal policy derived for the discounted problem in Section~\ref{subsec:struct_dis} also hold for the
average cost problem~\eqref{opt_prob}. This is achieved by considering a sequence of $\beta$-discounted optimal policies and showing that their limit, as $\beta \uparrow 1$, yields an average cost optimal policy with the same structure as shown in the following proposition. This method is commonly referred to as the vanishing discount approach~\cite{Hernandez2012discrete}.

\begin{proposition} \label{prop:opt_pol_convg}
    Consider the $\beta$-discounted MDP~\eqref{discount_opt_prob}. Let Assumption~\ref{ass:sys_chn} hold. Let $\{\beta_n\}_{n \in \bZ_+}$, $\beta_n \in (0,1)$ be any sequence of discount factors converging to $1$. Then, for any $s \in \cS$ there exists some subsequence $\{\beta_{n_k}\} \subseteq \{\beta_n\}$ such that $\beta_{n_k} \uparrow 1$, $n_k \in \bZ_+$, with $\{\pi^{\lf(\beta_{n_k}\rt)}(s)\}$ denoting the corresponding sequence of $\beta$-discounted optimal policies. Also, there exists a stationary policy $\pi\ust_{av.}$ such that $\lim_{k \rightarrow \infty} \pi\ust_{av.}(s) = \pi^{\lf(\beta_{n_k}\rt)}(s)$. Then, $\pi\ust_{av.}$ is an optimal policy for the average cost problem~\eqref{opt_prob}. Moreover, if $\zeta\ust$ is the optimal average cost~\eqref{opt_prob} and $V\ub$ is as in~\eqref{eq:V}, then we have $\zeta\ust = \lim_{\beta \rightarrow 1} (1-\beta) V\ub(s)$, $s \in \cS$.
\end{proposition}

\begin{proof}
    The proof follows from~\cite[Theorem 3.8]{schal1993average}. To apply this theorem, it suffices to verify that the General Assumption~\cite[p. 164]{schal1993average}, Assumptions (S) and (B) in~\cite{schal1993average} are satisfied in our setup. In our case, Assumption (S) holds trivially since the action set $\cU(s)$ is finite for each $s \in \cS$, see~\cite[p. 165]{schal1993average}. We show in Lemma~\ref{lemma:GA} that the General Assumption is satisfied. Next, we note that Assumption (B) is the same as Condition 5.4.5 (b) of~\cite{Hernandez2012discrete}. It then follows from~\cite[Theorem 5.4.6]{Hernandez2012discrete} that Assumption (B) holds once we verify that Assumptions 4.2.1 and 5.4.1 of~\cite{Hernandez2012discrete} are satisfied. These conditions are verified in Lemmas~\ref{lemma:VI_assump} and~\ref{lemma:cond_avg_cost}, respectively.
\end{proof}

Before stating the main result, we define four regions $R_i$, $i\in\{1,2,3,4\}$, for the average cost problem. Each $R_i$ is the limit of corresponding discounted cost problem region $R\ub_{i}$ as $\beta \rightarrow 1$, where $R\ub_i$, $i \in \{1,2,3,4\}$ are defined in Definition~\ref{def:4reg_dis}.

\begin{definition} \label{def:4reg_avg_cost}
    Consider the average cost MDP (cf.~\eqref{opt_prob}), and suppose that Assumption~\ref{ass:sys_chn} holds. Consider the parameter space $(a, d, p_{01}, p_{10}) \in \bR \times \{2,3,\ldots\} \times (0,1)^2$, where $a \in \bR$ is the AR process parameter~(cf.~\eqref{source}), $d \in \{2,3,\ldots\}$ is Channel 2 delay, and $(p_{01}, p_{10}) \in (0,1)^2$ are the parameters of Channel 1. We partition the admissible parameter space, defined as, $\cD := \{(a, d, p_{01}, p_{10}): a^2 (1-p_{01}) < 1\}$ into the following four regions, 
    \nal{
    R_1 &= \{(a,d,p_{01},p_{10}) : \\
    &\qquad F(a,d,p_{01},p_{10}) \le 0,\;
        H(a,d,p_{01},p_{10}) \le 0 \}, \\
    R_2 & = \{(a, d,p_{01}, p_{10}): \\
    &\qquad F(a, d,p_{01}, p_{10}) \le 0, H(a, d,p_{01}, p_{10}) > 0 \}, \\
    R_3 & = \{(a, d,p_{01}, p_{10}): \\
    & \qquad F(a, d,p_{01}, p_{10}) > 0, G(a, d, p_{01}, p_{10}) \leq 0 \}, \\
    R_4 & = \{(a, d,p_{01}, p_{10}): \\
    & \qquad F(a, d,p_{01}, p_{10}) > 0, G(a, d,p_{01}, p_{10}) > 0 \},
    }
    where the functions $F(\cdot)$, $H(\cdot)$, $G(\cdot) : \bR \times \{2,3,\ldots\} \times (0,1)^2 \mapsto \bR$ are defined as
    \nal{
    & F(a, d,p_{01}, p_{10}) = \sum_{i=0}^{\infty} \lf(a^2 (1 - p_{01})\rt)^i - \sum_{i=0}^{d-1} \lf(a^2\rt)^i \\
    & H(a, d,p_{01}, p_{10}) = 1 + a^2 p_{10} \sum_{i=0}^{\infty} \lf(a^2 (1 - p_{01})\rt)^i \\
    & \hspace{2.5cm} - \sum_{i=0}^{d-1} \lf(a^2\rt)^i \\
    & G(a, d,p_{01}, p_{10}) = 1 + a^2 p_{10} \sum_{i=0}^{d-1} \lf(a^2\rt)^i - \sum_{i=0}^{d-1} \lf(a^2\rt)^i.
    }
\end{definition}

The following is the main result of this section.

\begin{theorem} \label{thm:avg}
    Consider the average cost problem~\eqref{opt_prob}, and let Assumption~\ref{ass:sys_chn} hold. Then,  there exists an optimal policy for~\eqref{opt_prob}, denoted by $\pi\ust(\D,0,c)$, that satisfies the following: 

    \begin{enumerate}
        \item[1)] if $(a, d, p_{01}, p_{10}) \in R_1$, then $\pi\ust(\cdot,0,c)$ is a non-increasing threshold-type policy in $\D$ for each $c \in \{0,1\}$;

        \item[2)] if $(a, d, p_{01}, p_{10}) \in R_2$, then $\pi\ust(\cdot,0,0)$ is a non-increasing threshold-type policy and $\pi\ust(\cdot,0,1)$ is a non-decreasing threshold-type policy in $\D$;
        
        \item[3)] if $(a, d, p_{01}, p_{10}) \in R_3$, then $\pi\ust(\cdot,0,0)$ is a non-decreasing threshold-type policy and $\pi\ust(\cdot,0,1)$ is a non-increasing threshold-type policy in $\D$;

        \item[4)] if $(a, d, p_{01}, p_{10}) \in R_4$, then $\pi\ust(\cdot,0,c)$ is a non-decreasing threshold-type policy in $\D$ for each $c \in \{0,1\}$.
    \end{enumerate}
\end{theorem}

\begin{proof}
    We first consider a sequence of discount factors $\{\beta_n\}_{n \in \bZ_+}$, $\beta \in (0,1)$ such that $\beta_n \rightarrow 1$. We then have from Theorem~\ref{thm:discount} that for any $(a,d,p_{01},p_{10}) \in R\ub_i$, $i \in \{1,2,3,4\}$, there exists a corresponding non-increasing or non-decreasing threshold-type policy in $\D$ for each $c \in \{0,1\}$. The proof then follows from Proposition~\ref{prop:opt_pol_convg}, since the limit of a sequence of non-increasing (respectively, non-decreasing) threshold-type policies in $\D$ as $\beta \rightarrow 1$ remains a non-increasing (respectively, non-decreasing) threshold-type policy in $\D$. 
\end{proof}

\section{numerical results} \label{sec:num_res}

In this section, we first consider the case when the system parameters $(a,d,p_{01}, p_{10})$ are known. For this setting, we employ relative value iteration (RVI)~\cite{Hernandez2012discrete} to compute an optimal policy for the average cost problem~\eqref{opt_prob}, as discussed in Section~\ref{subsec:RVI}. In practice, however, the system parameters may not be known a priori. In such cases, we use the actor-critic (AC)~\cite{konda2003onactor} algorithm to learn an optimal policy, as described in detail in Section~\ref{subsec:AC}. In general, searching for an optimal policy over all possible policies can be very difficult because the policy space is large. However, we have shown in Theorem~\ref{thm:avg} that there exists an optimal policy that is either non-increasing or non-decreasing threshold-type. This implies that we only need to search within a much smaller class of policies. We use this structural result to design the AC algorithm which makes the learning process faster and more efficient. Throughout this section, we truncate the estimation error $\D$ to the interval $[-15,15]$ and define the truncated state-space as $\Tilde{\cS} = [-15,15] \times \{0,1,\ldots,d-1\} \times \{0,1\}$. 

\subsection{Known Parameter Values} \label{subsec:RVI}

We use the RVI algorithm to obtain an optimal policy. It follows from~\cite[Theorem 5.5.4]{Hernandez2012discrete} that under Assumptions 4.2.1 and 5.5.1, there exist a constant $\zeta\ust$ and a continuous function $h: \cS \rightarrow \bR$ that satisfy the following average cost optimality equation (ACOE) for all $s \in \cS$:
\al{
\zeta\ust + h(s) = \min_{u \in \cU} \lf[\ell(s,u) + \bE_{s_+ \sim p(\cdot |s;u)}[h(s_+)]\rt], \label{eq:ACOE}
}
where $\ell$ is as defined in~\eqref{instan_cost}, $\zeta\ust$ is the optimal average cost of~\eqref{opt_prob}, and $h$ is the relative value function~\cite[Chapter 5]{Hernandez2012discrete}. Moreover, there exists a deterministic stationary optimal policy $\Tilde{\pi}$ that attains the minimum on the right-hand side of~\eqref{eq:ACOE}. We will now use RVI algorithm~\cite[p. 101]{Hernandez2012discrete} to solve~\eqref{eq:ACOE}. For the RVI implementation, we discretize the truncated state space of the estimation error, $[-15, 15]$, with a quantization width of $0.1$. The relative value iteration (RVI) algorithm is shown in Algorithm~\ref{alg:rvi}. The optimal $n$-stage cost~\cite[p. 101]{Hernandez2012discrete} is denoted by $j_n$ and is recursively updated as shown in~\eqref{eq:j_n_1}--\eqref{eq:j_n_2}. $\Tilde{\pi}_n$ in~\eqref{eq:pi_n_1} denotes the minimizer of the right hand side of~\eqref{eq:j_n_1}. Then, it follows from~\cite[Theorem 5.6.3]{Hernandez2012discrete} that under Assumptions 4.2.1, 5.5.1, and 5.6.1 of~\cite{Hernandez2012discrete}, we have that $h_n(s) \rightarrow h(s)$ (see \eqref{eq:h_n}) and $\zeta_n(s) \rightarrow \zeta\ust$ (see \eqref{eq:zeta_n}) for all $s \in \cS$. Figs.~\ref{fig:r1optimalpoliciesscatterv2},~\ref{fig:r2optimalpoliciesscatterv2},~\ref{fig:r3optimalpoliciesscatterv2}, and~\ref{fig:r4optimalpoliciesscatterv2} show the optimal policies for the average cost problem~\eqref{opt_prob} obtained using RVI for regions $R_1, R_2, R_3$, and $R_4$, respectively. It can be seen that, (i) in Fig.~\ref{fig:r1optimalpoliciesscatterv2} for $R_1$, $\pi\ust(\D, 0, c)$ are non-increasing threshold-type policies in $\D$ for each $c \in \{0,1\}$, (ii) in Fig.~\ref{fig:r2optimalpoliciesscatterv2} for $R_2$, $\pi\ust(\D, 0, 0)$ is non-increasing threshold-type policy and $\pi\ust(\D,0,1)$ is non-decreasing threshold-type policy in $\D$, (iii) in Fig.~\ref{fig:r3optimalpoliciesscatterv2} for $R_3$, $\pi\ust(\D, 0, 0)$ is non-decreasing threshold-type policy and $\pi\ust(\D, 0, 1)$ is non-increasing threshold-type policy in $\D$, and (iv) in Fig.~\ref{fig:r4optimalpoliciesscatterv2} for $R_4$, $\pi\ust(\D, 0, c)$ are non-decreasing threshold-type policies in $\D$ for each $c \in \{0,1\}$. These observations experimentally validate our result in Theorem~\ref{thm:avg}. We next compare the performance of an optimal policy for the average cost problem~\eqref{opt_prob} obtained via RVI with the following three suboptimal policies:

\begin{enumerate}
	\item Suboptimal policy 1: The sensor always transmits a data packet via Channel 1.
	
	\item Suboptimal policy 2: The sensor always attempts a data packet transmission using Channel 2 whenever it is available.
	
	\item Suboptimal policy 3: When Channel 2 is available, the sensor chooses Channel 1 for packet transmission with probability $p^{(pol.)}$ and Channel 2 with probability $1-p^{(pol.)}$.
\end{enumerate}

\begin{algorithm}[htb]
\caption{Relative Value Iteration (RVI)}
\label{alg:rvi}
\begin{algorithmic}[0]

\State Initialize $j_0(s)=0$, $h_0(s)=0$ for all $s \in \Tilde{\mathcal{S}}$
\State Choose reference state $(s_0)$ and tolerance $\epsilon$

\For{$n = 0,1,2,\dots$}
\ForAll{$s \in \Tilde{\cS}$}

\State Compute

\If{$r \ge 1$} \State \begin{equation} \begin{aligned} j_{n+1}(s) &= \ell(s,0) \\ &\quad + \int_{\tilde{\mathcal S}} p(s_+|s;0)\, j_n(s_+)\, ds_+ \end{aligned} \label{eq:j_n_1} \end{equation} 

\ElsIf{$r = 0$} \State \begin{equation} \begin{aligned} j_{n+1}(s) &= \min_{u\in\{1,2\}} \Big\{ \ell(s,u) \\ &\quad + \int_{\tilde{\mathcal S}} p(s_+|s;u)\, j_n(s_+)\, ds_+ \Big\} \end{aligned} \label{eq:j_n_2} \end{equation} 

\EndIf

\State Select
\If{$r = 0$}
\State
\begin{equation}
	\begin{aligned}
		\tilde \pi_{n+1}(s)
		&\in
		\arg\min_{u\in\{1,2\}}
		\Big\{\ell(s,u) \\
		& +
		\int_{\tilde{\cS}} p(s_+|s;u)\,
		j_n(s_+)\,ds_+
		\Big\} \label{eq:pi_n_1}
	\end{aligned}
\end{equation}
\ElsIf{$r \geq 1$} \State \begin{equation} \begin{aligned} \tilde \pi_{n+1}(s) = 0 \end{aligned} \label{eq:pi_n_2} \end{equation}
\EndIf

\State \begin{equation} h_{n+1}(s)=j_{n+1}(s)-j_{n+1}(s_0) \label{eq:h_n} \end{equation} 

\State \begin{equation} \zeta_{n+1}(s)=j_{n+1}(s)-j_n(s) \label{eq:zeta_n} \end{equation}

\EndFor

\If{$\|h_{n+1}-h_n\|_\infty<\epsilon$}
\State \textbf{break}
\EndIf

\EndFor

\State \textbf{return} $\zeta_{n+1}(s)$ and $\tilde \pi_n$

\end{algorithmic}
\end{algorithm}

Fig.~\ref{fig:r1performance} compares the performance of the optimal policy with suboptimal policies 1, 2, and 3 for region $R_1$ as (a) Markovian Channel 1 parameter $p_{10}$, (b) the AR process parameter $a$, (c) the transmission energy unit $\lambda$, (d) Channel 2 delay $d$ are varied, and (e) the suboptimal policy 3 parameter $p^{(pol.)}$ are varied, respectively. It can be seen from Fig.~\ref{fig:r1performance}(a) that as $p_{10}$ increases, Markovian Channel 1 becomes more unreliable, and hence the average cost increases. Fig.~\ref{fig:r1performance}(b) indicates that as $|a|$ increases, the estimation error grows faster according to~\eqref{e_evolve}, again increasing the average cost. Similar increases occur as $\lambda$ or $d$ increases, as depicted in Figs.~\ref{fig:r1performance}(c) and (d), respectively. Fig.~\ref{fig:r1performance}(e) shows that as $p^{(pol.)}$ is varied from $0$ to $1$, the randomized policy, i.e., suboptimal policy 3, interpolates between always using Channel 2 and always using Channel 1. Moreover, it can be seen that the optimal policy outperforms all suboptimal policies. Analogous comparisons for regions $R_2$, $R_3$, and $R_4$ are shown in Figs.~\ref{fig:r2performance}, \ref{fig:r3performance}, and \ref{fig:r4performance}, respectively.

\begin{figure}
	\centering
	\includegraphics[width=1\linewidth]{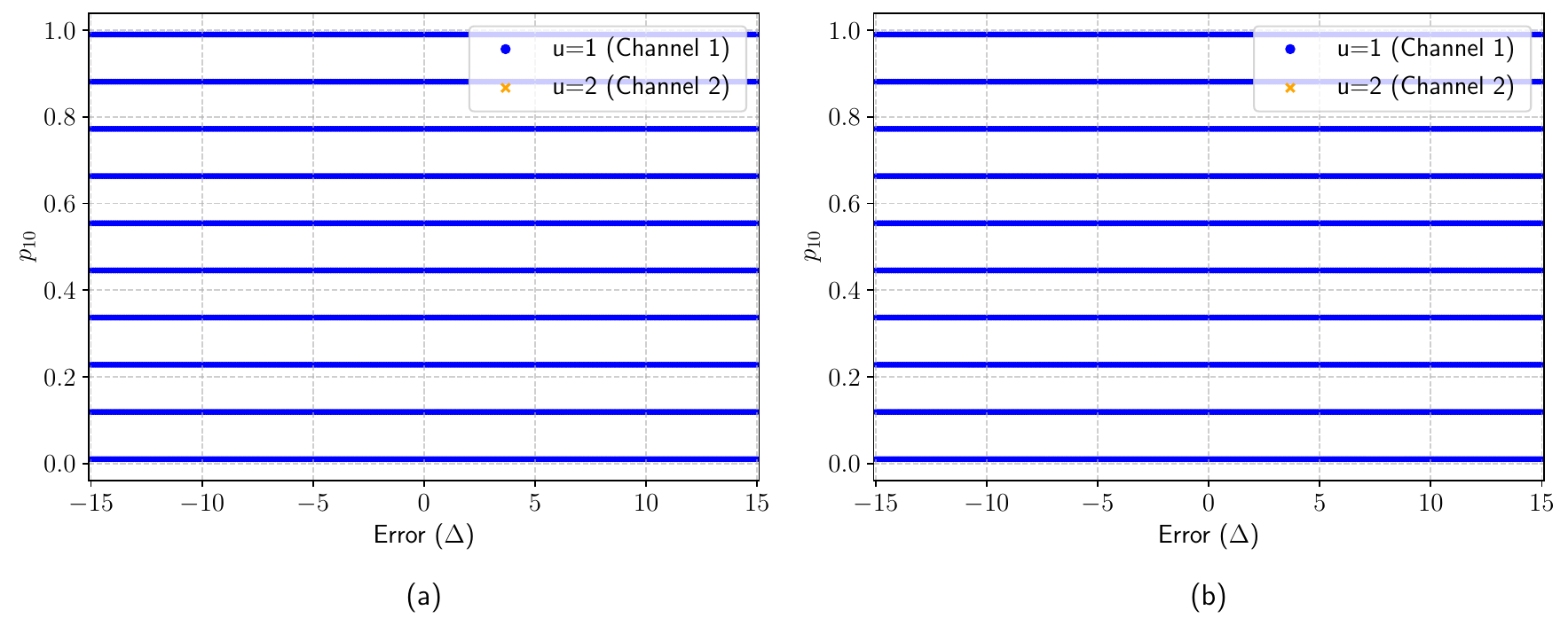}
	\caption{Optimal policy obtained using RVI for region $R_1$ as $p_{10}$ varies with fixed parameters $p_{01} = 0.6, a = 0.9, \lambda = 0.6, d= 4$: (a) $c = 0$; (b) $c = 1$.}
	\label{fig:r1optimalpoliciesscatterv2}
\end{figure}

\begin{figure}
	\centering
	\includegraphics[width=1\linewidth]{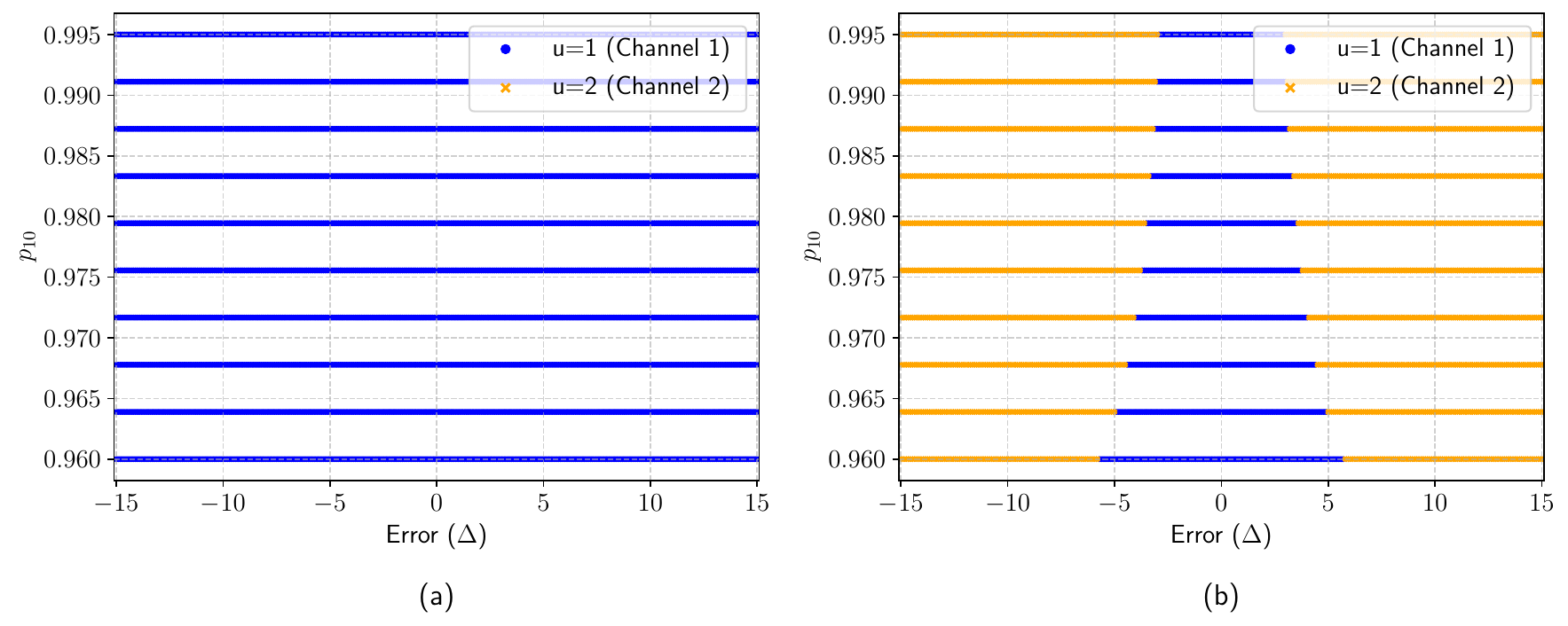}
	\caption{Optimal policy obtained using RVI for region $R_2$ as $p_{10}$ varies with fixed parameters $p_{01} = 0.24, a = 0.9, \lambda = 0.6, d= 4$: (a) $c = 0$; (b) $c = 1$.}
	\label{fig:r2optimalpoliciesscatterv2}
\end{figure}

\begin{figure}
	\centering
	\includegraphics[width=1\linewidth]{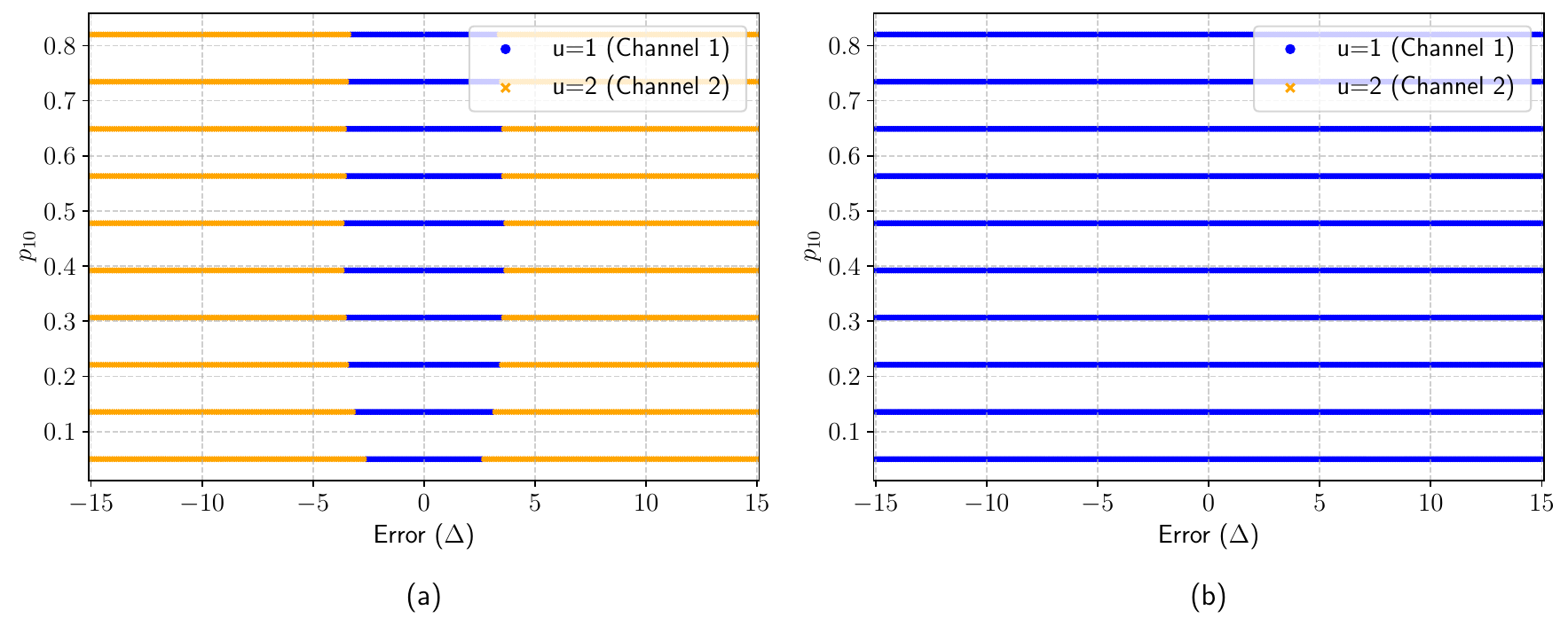}
	\caption{Optimal policy obtained using RVI for region $R_3$ as $p_{10}$ varies with fixed parameters $p_{01} = 0.17, a = 0.9, \lambda = 0.6, d= 4$: (a) $c = 0$; (b) $c = 1$.}
	\label{fig:r3optimalpoliciesscatterv2}
\end{figure}

\begin{figure}
	\centering
	\includegraphics[width=1\linewidth]{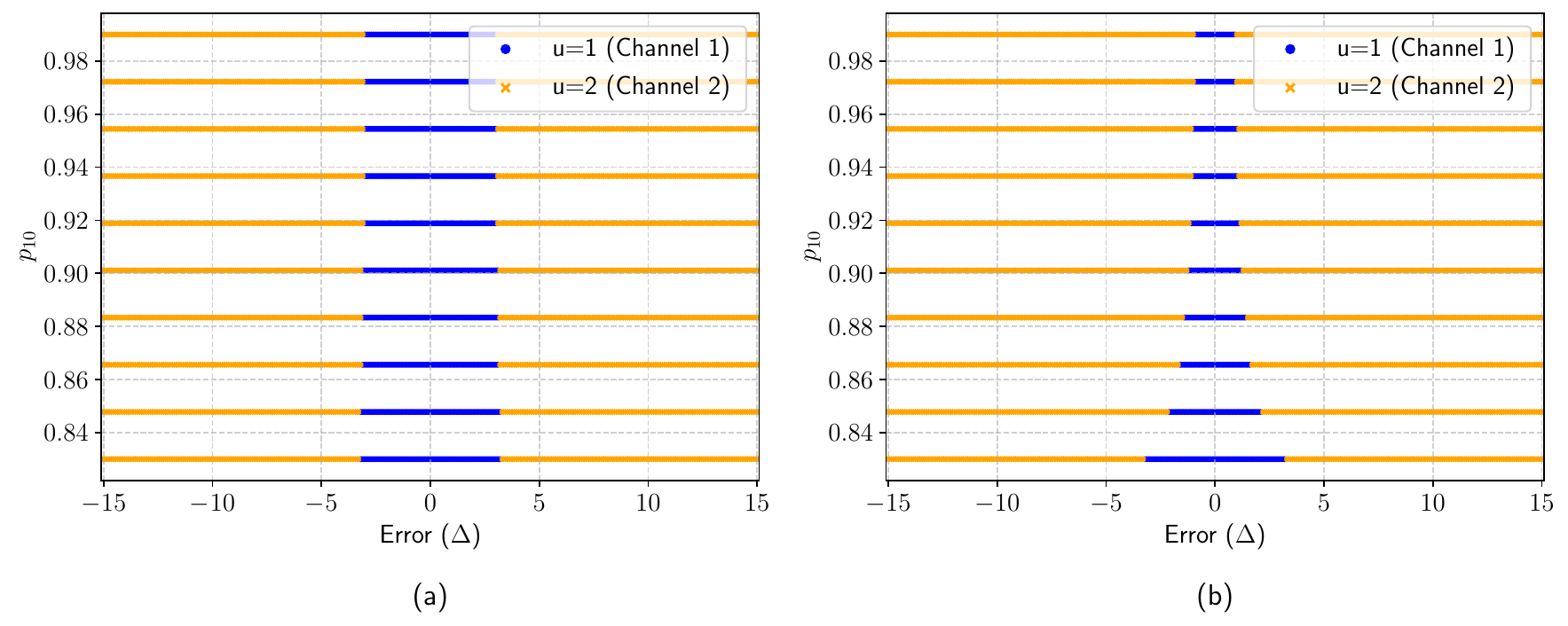}
	\caption{Optimal policy obtained using RVI for region $R_4$ as $p_{10}$ varies with fixed parameters $p_{01} = 0.24, a = 0.9, \lambda = 0.6, d= 4$: (a) $c = 0$; (b) $c = 1$.}
	\label{fig:r4optimalpoliciesscatterv2}
\end{figure}

\begin{figure}
	\centering
	\includegraphics[width=1\linewidth]{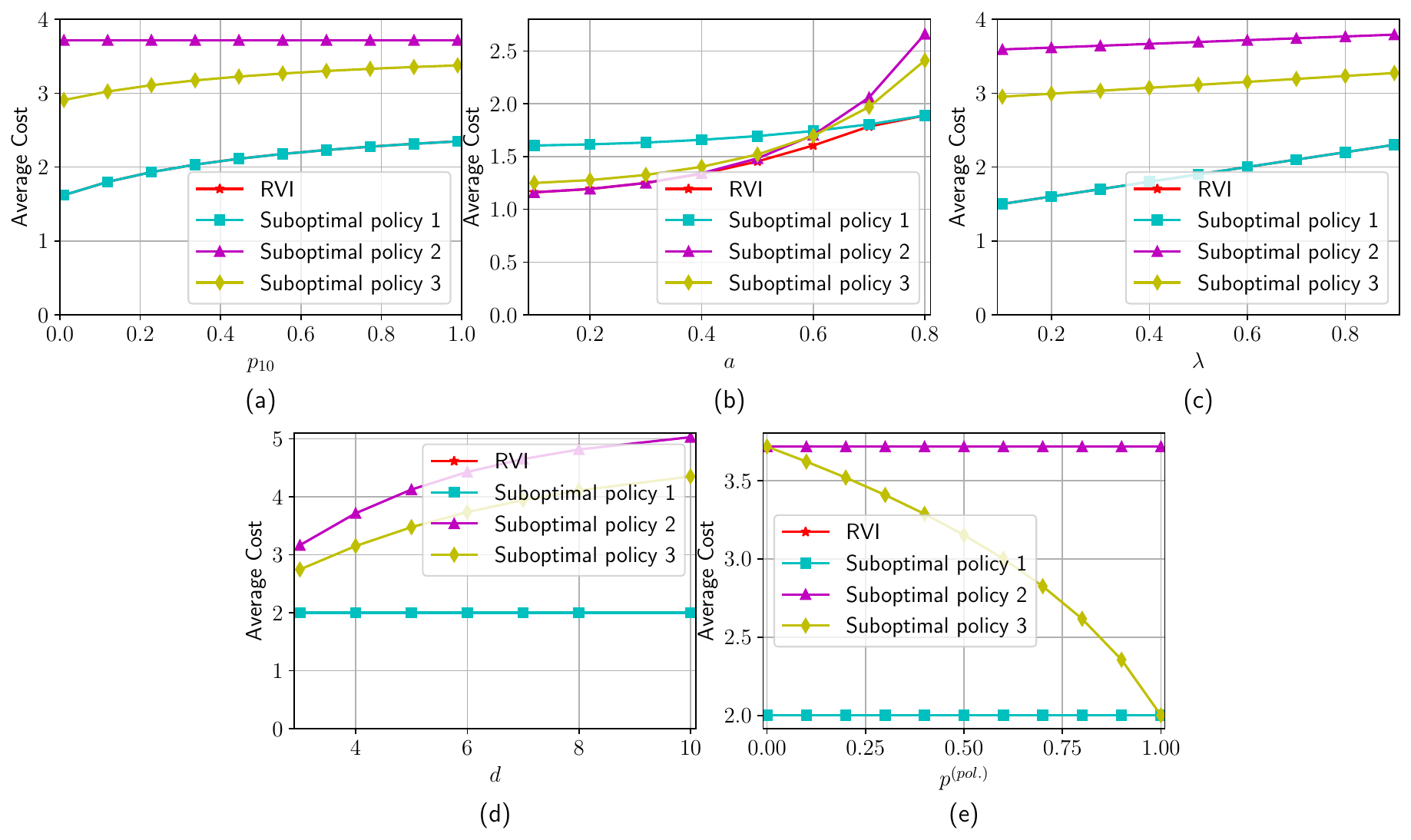}
	\caption{Performance comparison of an optimal policy obtained using RVI and three suboptimal policies in region $R_1$ as (a) $p_{10}$ is varied with $p_{01} = 0.6$, $a = 0.9$, $\lambda = 0.6$, $d = 4$; (b) $a$ is varied with $p_{01} = 0.6$, $p_{10} = 0.3$, $\lambda = 0.6$, $d = 4$; (c) $\lambda$ is varied with $p_{01} = 0.6$, $p_{10} = 0.3$, $a = 0.9$, $d = 4$; (d) $d$ is varied with $p_{01} = 0.6$, $p_{10} = 0.3$, $a = 0.9$, $\lambda = 0.6$; (e) $p^{(pol.)}$ is varied with $p_{01} = 0.6$, $p_{10} = 0.3$, $a = 0.9$, $\lambda = 0.6$, $d = 4$. \textbf{Note}: The RVI plot (in red) is not visible in plots (a),(c), (d), and (e) since the suboptimal policy 1 coincides with the optimal policy, and hence overlaps the red line for RVI.}
	\label{fig:r1performance}
\end{figure}

\begin{figure}
	\centering
	\includegraphics[width=1\linewidth]{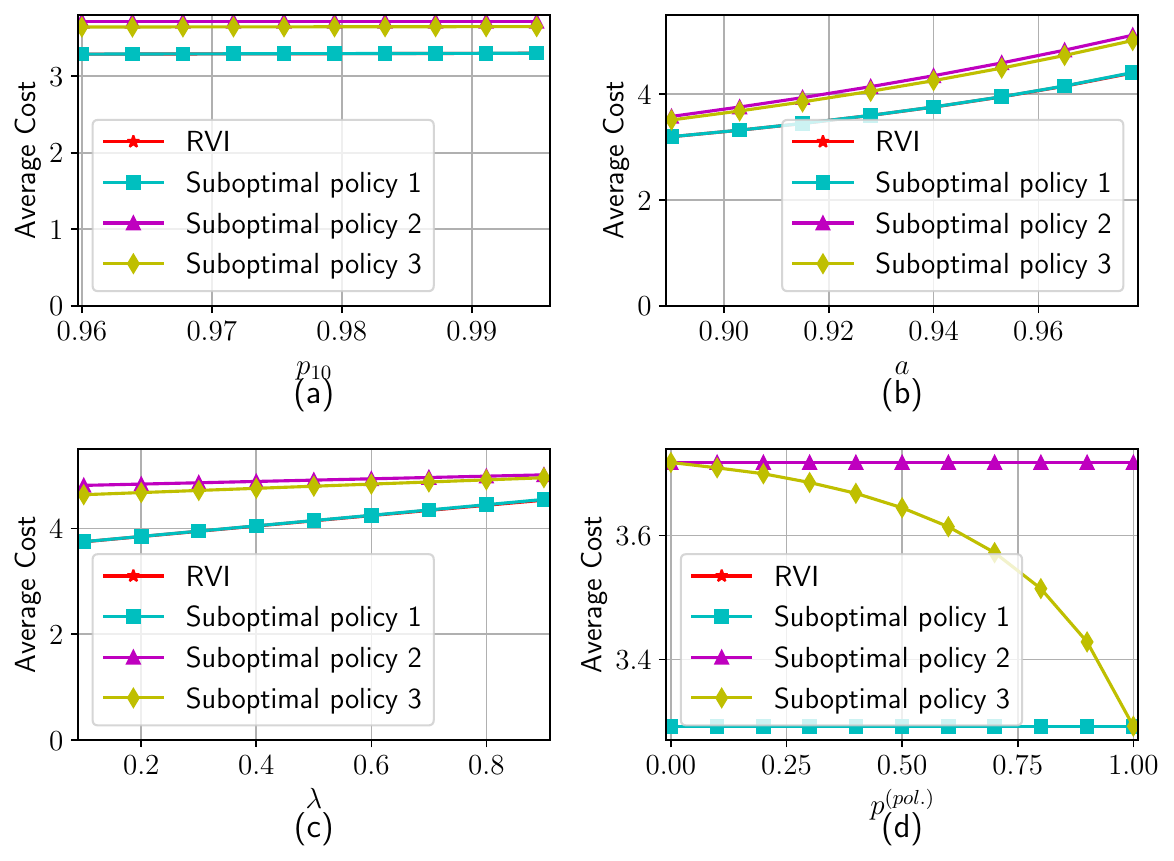}
	\caption{Performance comparison of an optimal policy obtained using RVI and three suboptimal policies in region $R_2$ as (a) $p_{10}$ is varied with $p_{01} = 0.24$, $a = 0.9$, $\lambda = 0.6$, $d = 4$; (b) $a$ is varied with $p_{01} = 0.24$, $p_{10} = 0.97$, $\lambda = 0.6$, $d = 4$; (c) $\lambda$ is varied with $p_{01} = 0.24$, $p_{10} = 0.97$, $a = 0.97$, $d = 4$; (d) $p^{(pol.)}$ is varied with $p_{01} = 0.24$, $p_{10} = 0.97$, $a = 0.9$, $\lambda = 0.6$, $d = 4$. \textbf{Note}: Unlike the other regions, no subplot varying $d$ is shown here. This is because for fixed $a, p_{01}, p_{10}$, at most one value of $d$ satisfies the conditions defining $R_2$ in Definition~3.}
	\label{fig:r2performance}
\end{figure}

\begin{figure}
	\centering
	\includegraphics[width=1\linewidth]{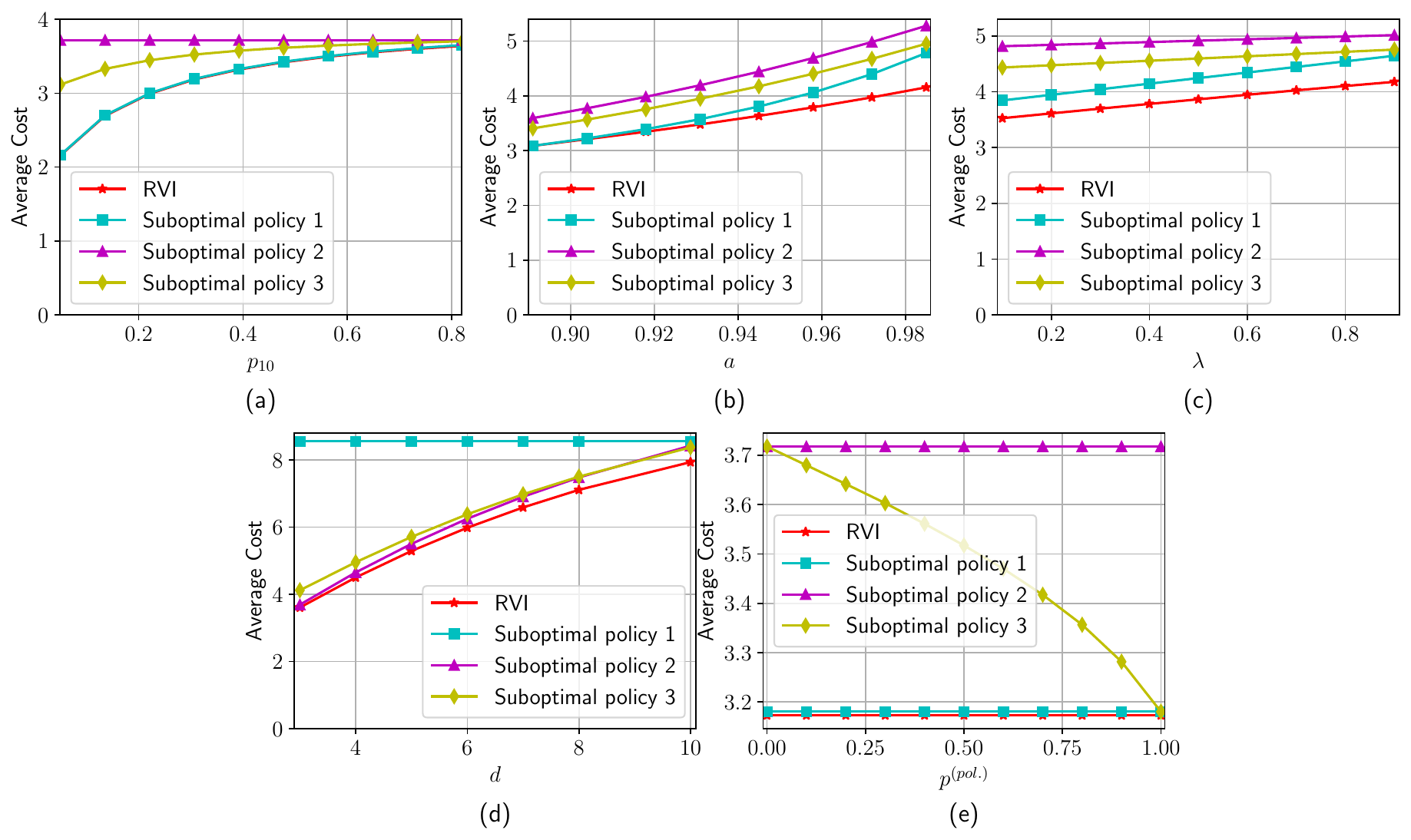}
	\caption{Performance comparison of an optimal policy obtained using RVI and three suboptimal policies in region $R_3$ as (a) $p_{10}$ is varied with $p_{01} = 0.17$, $a = 0.9$, $\lambda = 0.6$, $d = 4$; (b) $a$ is varied with $p_{01} = 0.17$, $p_{10} = 0.3$, $\lambda = 0.6$, $d = 4$; (c) $\lambda$ is varied with $p_{01} = 0.17$, $p_{10} = 0.3$, $a = 0.97$, $d = 4$; (d) $d$ is varied with $p_{01} = 0.04$, $p_{10} = 0.55$, $a = 0.956$, $\lambda = 0.6$; (e) $p^{(pol.)}$ is varied with $p_{01} = 0.17$, $p_{10} = 0.3$, $a = 0.9$, $\lambda = 0.6$, $d = 4$. \textbf{Note}: Unlike the other regions, no subplot varying $d$ is shown here. This is because for fixed $a, p_{01}, p_{10}$, at most one value of $d$ satisfies the conditions defining $R_2$ in Definition~3.}
	\label{fig:r3performance}
\end{figure}

\begin{figure}
	\centering
	\includegraphics[width=1\linewidth]{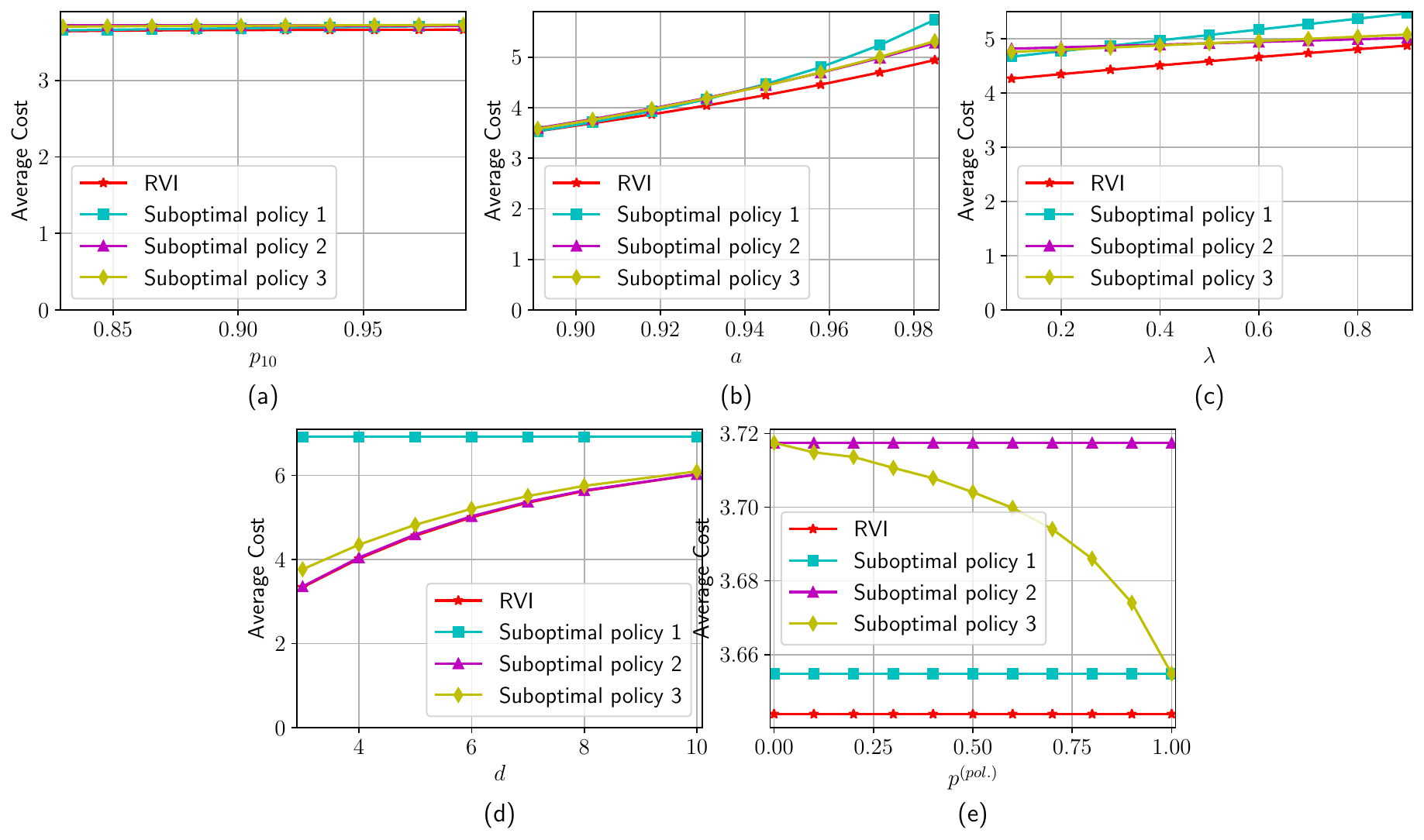}
	\caption{Performance comparison of an optimal policy obtained using RVI and three suboptimal policies in region $R_4$ as (a) $p_{10}$ is varied with $p_{01} = 0.17$, $a = 0.9$, $\lambda = 0.6$, $d = 4$; (b) $a$ is varied with $p_{01} = 0.17$, $p_{10} = 0.83$, $\lambda = 0.6$, $d = 4$; (c) $\lambda$ is varied with $p_{01} = 0.17$, $p_{10} = 0.83$, $a = 0.97$, $d = 4$; (d) $d$ is varied with $p_{01} = 0.01$, $p_{10} = 0.96$, $a = 0.922$, $\lambda = 0.6$; (e) $p^{(pol.)}$ is varied with $p_{01} = 0.17$, $p_{10} = 0.83$, $a = 0.9$, $\lambda = 0.6$, $d = 4$.}
	\label{fig:r4performance}
\end{figure}

\subsection{Actor-Critic Algorithm} \label{subsec:AC}
\begin{figure}
	\centering
	\includegraphics[width=1\linewidth]{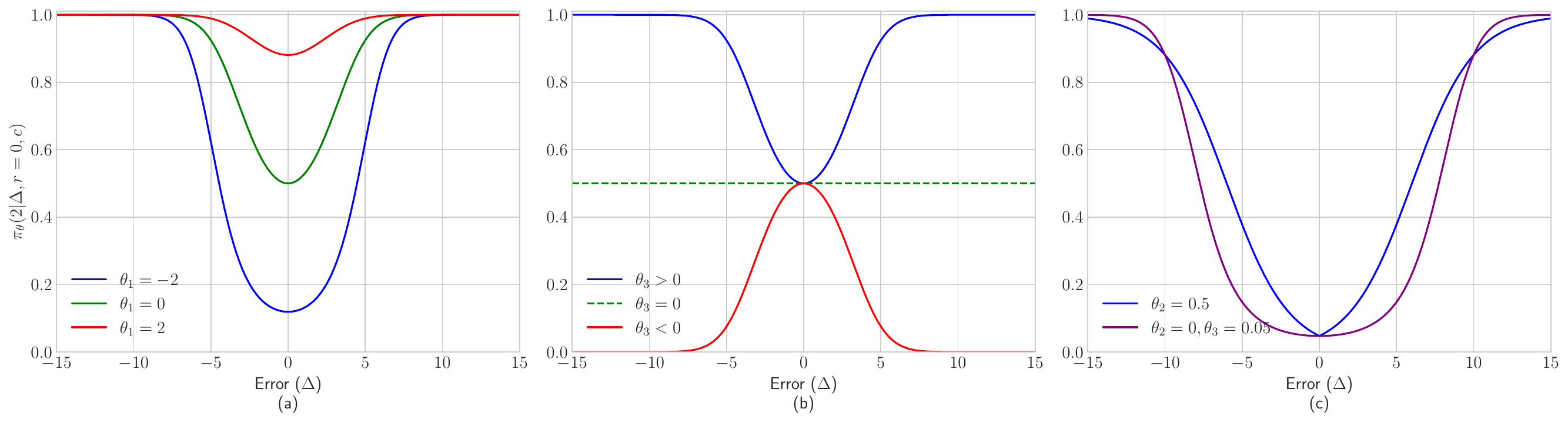}
	\caption{Effect of the policy parameter $\te$ on $\pi_{\te}(2|\D,r=0,c)$. (a) The parameters $\te_1, \te_4$ control the probability of transmitting via Channel 2; (b) the parameters $\te_3, \te_6$ enables the policy $\pi_{\te}(2|\D,r=0,c)$ to capture both non-decreasing and non-increasing threshold-type policy structures; (c) the parameter $\te_2, \te_5$ captures the steepness of the decision boundary.}
	\label{fig:softmax_explain}
\end{figure}

In this section, we introduce the AC algorithm~\cite{konda2003onactor} that learns an optimal policy for the average cost problem~\eqref{opt_prob} when the system parameters ($a, p_{01}, p_{10}, d$) are unknown. Leveraging the structural properties established in Theorem~\ref{thm:avg}, we construct a parameterized class of stationary randomized policies that can represent the threshold-type structures in regions $R_1,\ldots,R_4$. The parameters of this policy are then updated iteratively using the AC algorithm in a direction that improves performance, i.e., reduces the long-run average cost. For $\Delta\in[-15,15]$, let
$\theta=(\theta_1,\ldots,\theta_6)\in\bR^6$ and define
\nal{
& \pi_{\te}(2|\D, r = 0, c=0) = \frac{1}{1 + \exp(-(\te_1 + \te_2 |\D| + \te_3 \D^2))}, \\
& \pi_{\te}(2|\D, r = 0, c = 1) = \frac{1}{1 + \exp(-(\te_4 + \te_5 |\D| + \te_6 \D^2))},
}
where $\pi_{\theta}(u|\Delta, r, c)$ denotes the probability of selecting action $u$ when the current state is $(\Delta, r, c)$. Then, $\pi_{\te}(1|\D, r = 0, c) = 1 - \pi_{\te}(2|\D, r = 0, c) $. The role of $\te$ in the policy parameterization is shown in Fig.~\ref{fig:softmax_explain}. For $r\geq 1$, Channel 2 is busy and the action is fixed as $u=0$, so $\pi_{\theta}(0|\Delta,r,c)=1$ for all $r \geq 1$, $\Delta \in [-15,15]$, and $c \in \{0,1\}$. Let $\zeta(\pi_{\te})$ denote the average cost under policy $\pi_{\te}$, i.e.,
\nal{
\zeta(\pi_{\te}) = \liminf_{T \rightarrow \infty} \frac{1}{T} \bE_{\pi_{\te}} \lf[\sum_{t = 0}^{T} \ell(s(t),u(t))\rt],
}
where $\ell$ is as in~\eqref{instan_cost} and $\bE_{\pi_{\te}}$ denotes the expectation with respect to the measure induced by the policy $\pi_{\te}$. Next, for every $\te \in \bR^6$, we define the advantage function~\cite{konda2003onactor} $\rho_{\te}: \Tilde{\cS} \times \{0,1,2\} \rightarrow \bR$ as follows:
\nal{
\rho_{\te}(s,u) = \ell(s,u) - \zeta(\pi_{\te}) + \bE_{\pi_{\te}} \lf[h(s_+) \mid s,u\rt],
} 
where $\ell$ is as in~\eqref{instan_cost}, $s_+ \in \Tilde{\cS}$ denotes the next state given that the current state is $s$ and the current action is $u$, and $h$ is as in~\eqref{eq:ACOE}. The critic estimates $\rho_{\te}$, and the actor uses this estimate to update $\te$. Since the state space $\Tilde{\cS}$ is continuous, we employ linear function approximation~\cite{konda2003onactor} for the critic to estimate $\rho_{\te}$.~The critic computes an approximation of $\rho_{\te}$, which is denoted by $\Tilde{\rho}_{\te}$ and is described below for $s \in \Tilde{\cS}$, $u \in \{0,1,2\}$.
\nal{
\Tilde{\rho}_{\te}^{(w)}(s,u) = \sum_{i=1}^{n^{(critic)}} \omega_i \phi_{\te}^{(i)}(s,u),
}
where $n^{(critic)} \in \bN$ denotes the number of critic parameters, $\omega = (\omega_1, \omega_2, \ldots, \omega_{n^{(critic)}})^T$ is the critic weight vector, and $\phi_{\te} = (\phi_{\te}^{(1)}, \phi_{\te}^{(2)}, \ldots, \phi_{\te}^{(n^{(critic)})})^{T}$ denotes  the feature vector used by the critic which depends on the actor parameter vector $\te$. Following~\cite{konda2003onactor}, which allows the number of critic parameters $n^{(critic)}$ to exceed the number of actor parameters, we choose the critic feature vector $\phi_{\te}$ as described below with $n^{(critic)} = 13$. For $\D \in [-15,15]$, $c \in \{0,1\}$, $r = 0$, $u \in \{1,2\}$, 
\nal{
	& \phi_{\theta}(\Delta,r,c,u) \\
	& =
	\begin{cases}
		\begin{aligned}
			\big[ &\nabla_{\theta}\log\pi_{\theta}(u|\Delta,r,c), \mathbf{0}_7 \big]^T
		\end{aligned}
		& \text{if } r=0,\; u\in\{1,2\}, \\[6pt]
		\begin{aligned}
			\big[ &\mathbf{0}_6, 1, \Delta, \Delta^2, c, r, r\Delta^2, \\
			& c\log(1+|\Delta|) \big]^T
		\end{aligned}
		& \text{if } r\ge1,\; u=0, 
	\end{cases}
}
where $\nabla_{\te}$ denotes the gradient with respect to $\te$, and $\mathbf{0}_n$ denotes the $n \times 1$ zero vector.

The pseudocode for the AC algorithm is given in Algorithm~\ref{algo:AC}. At each time $t$, the critic maintains the weight vector $\omega(t)$ (cf.~\eqref{eq:critic_update}), an
average cost estimate $\hat{\zeta}(t)$ (cf.~\eqref{eq:zeta}), and the eligibility trace $z(t)$~\cite{sutton1998reinforcement} (cf.~\eqref{eq:z}), where $\Tilde{s} \in \Tilde{\cS}$ is a fixed recurrent reference state chosen prior to learning. The actor maintains
the parameter vector $\theta(t)$ (cf.~\eqref{eq:actor_update}) at time $t$. Then, under Assumptions 3.3, 4.1, 4.2, 4.4, 4.5, 4.8, and 4.9 of~\cite{konda2003onactor}, the AC algorithm satisfies the following convergence result, which is shown in~\cite[Theorem 6.3]{konda2003onactor}.

\begin{theorem}
	Consider the AC algorithm presented in Algorithm~\ref{algo:AC} with TD(1) critic and linear function approximation. Then, $\liminf_t |\nabla_{\theta} \zeta(\pi_{\theta(t)})| = 0$ with probability 1, where $\zeta(\cdot)$ is the average cost function, and $\nabla_{\theta} \zeta(\pi_{\theta(t)})$ represents the gradient of $\zeta(\pi_{\theta})$ with respect to $\theta$, evaluated at $\theta(t)$.
\end{theorem}

\begin{algorithm}
	\caption{Actor-Critic (AC) Algorithm}
	\label{algo:AC}
	\begin{algorithmic}[0]
		
		\Require Policy parameterization $\pi_{\theta}(u|s)$
		\Require Advantage function approximation $\tilde{\rho}^{(\omega)}_{\theta}(s,u)$
		\Require Step sizes $\alpha^{(\theta)} > 0$, $\alpha^{(\omega)} > 0$
		
		\State Initialize $\theta(0)$, $\omega(0)$
		\State Initialize state $s(0) \in \Tilde{\mathcal S}$ and average cost estimate $\hat{\zeta}(0)$
		\State Sample $u(0) \sim \pi_{\theta(0)}(\cdot|s(0))$
		\State Set $z(0) = \phi_{\theta(0)}(s(0),u(0))$
		
		\For{$t = 1,2,\ldots,T$}
		
		\State Observe cost $\ell(s(t-1),u(t-1))$ (cf.~\eqref{instan_cost})
		\State Observe next state $(s(t))$ (cf.~\eqref{e_evolve})
		
		\State Sample $u(t) \sim \pi_{\theta(t-1)}(\cdot|s(t))$
		
		\State Update average cost estimate
		\al{
		\hat{\zeta}(t) & = \hat{\zeta}(t-1) \notag\\
		& +
		\alpha^{(\omega)}_t \bigl[\ell(s(t-1),u(t-1)) - \hat{\zeta}(t-1)
		\bigr] \label{eq:zeta}
		}	
		
		\State Update critic
		\al{\label{eq:critic_update}
		\omega(t) = \omega(t-1) + \alpha^{(\omega)}_t \delta(t) z(t-1)
		}
		where
		\[
		\begin{aligned}
			\delta(t) = &\ \ell(s(t-1),u(t-1)) - \hat{\zeta}(t-1) \\
			& + \tilde{\rho}^{(\omega(t-1))}_{\theta(t-1)}(s(t),u(t)) \\
			& - \tilde{\rho}^{(\omega(t-1))}_{\theta(t-1)}(s(t-1),u(t-1))
		\end{aligned}
		\]
		
		\State Update eligibility trace $z(t)$ (TD(1) critic) with $\Tilde{s} \in \Tilde{\cS}$
		\al{\label{eq:z}
		z(t)=
		\begin{cases}
			z(t-1) + \phi_{\theta(t-1)}(s(t),u(t)) 
			& \text{if } s(t) \neq \tilde s, \\
			\phi_{\theta(t-1)}(s(t),u(t)) & \text{otherwise}
		\end{cases}
	}
		
		\State Update actor
		\al{
		\theta(t) =~& \theta(t-1) -
		\alpha^{(\theta)}_t \tilde{\rho}^{(\omega(t-1))}_{\theta}(s(t),u(t)) \notag \\
		& \times  \nabla_{\te} \log \pi_{\te(t-1)}(u(t) \mid s(t)) \label{eq:actor_update}
	}
		
		\EndFor
		
	\end{algorithmic}
\end{algorithm}

We next present simulation results for the AC algorithm. The step-size sequences are chosen as $\alpha_t^{(\omega)} = 1/t^{0.6}$ and $\alpha_t^{(\theta)} = 1/t^{0.7}$. This choice ensures that the actor iterates evolve on a slower time-scale relative to the critic (cf.~\cite[Assumption 3.3]{konda2003onactor}), thereby yielding a standard two time-scale stochastic approximation scheme~\cite{bertsekas2012dynamic}. All reported results for the AC are obtained by averaging over $5$ independent runs and we plot the mean along with $\pm$ one standard deviation. In Fig.~\ref{fig:acrvivir1performancev4}, we compare the performance of the AC algorithm with RVI and VI for the $\beta$-discounted cost as described in Section~\ref{subsec:VI} for region $R_1$. The comparison is carried out by varying the parameters $p_{10}$, $a$, $\lambda$, and $d$ for $R_1$. We expect  the optimal policy obtained via RVI to outperform the AC algorithm. This is because (i) unlike AC, RVI has access to the true system parameters, and (ii) the AC algorithm relies on linear function approximation, which introduces approximation error. Nevertheless, as seen from Fig.~\ref{fig:acrvivir1performancev4}, the AC algorithm exhibits satisfactory performance. Moreover, it can be seen that as $\beta \uparrow 1$, the performance of VI approaches that of RVI, consistent with Proposition~\ref{prop:opt_pol_convg}. Similar trends for regions $R_2$, $R_3$, and $R_4$  are shown  in Figs.~\ref{fig:acrvivir2performancev4},~\ref{fig:acrvivir3performancev4},~\ref{fig:acrvivir4performancev4}, respectively.

Fig.~\ref{fig:softmaxacpolicystructure} shows the policy learned by the AC algorithm for the regions $R_1,\ldots,R_4$. We observe that the learned policies are consistent with the structural properties established in Theorem~\ref{thm:avg}. In particular, for $R_1$, the probability of selecting action $u=2$ remains nearly constant and is negligible for both $c \in \{0,1\}$. For $R_2$, the policy is decreasing in $|\D|$ for $c=0$ and increasing in $|\D|$ for $c=1$. For $R_3$, the reverse trend is observed, with the policy increasing in $|\D|$ for $c=0$ and decreasing for $c=1$. Finally, for $R_4$, the probability of selecting $u=2$ increases with $|\D|$ for each $c \in \{0,1\}$. These observations indicate that the AC algorithm is able to capture the threshold-type structure predicted by Theorem~\ref{thm:avg}.

\begin{figure}
	\centering
	\includegraphics[width=1\linewidth]{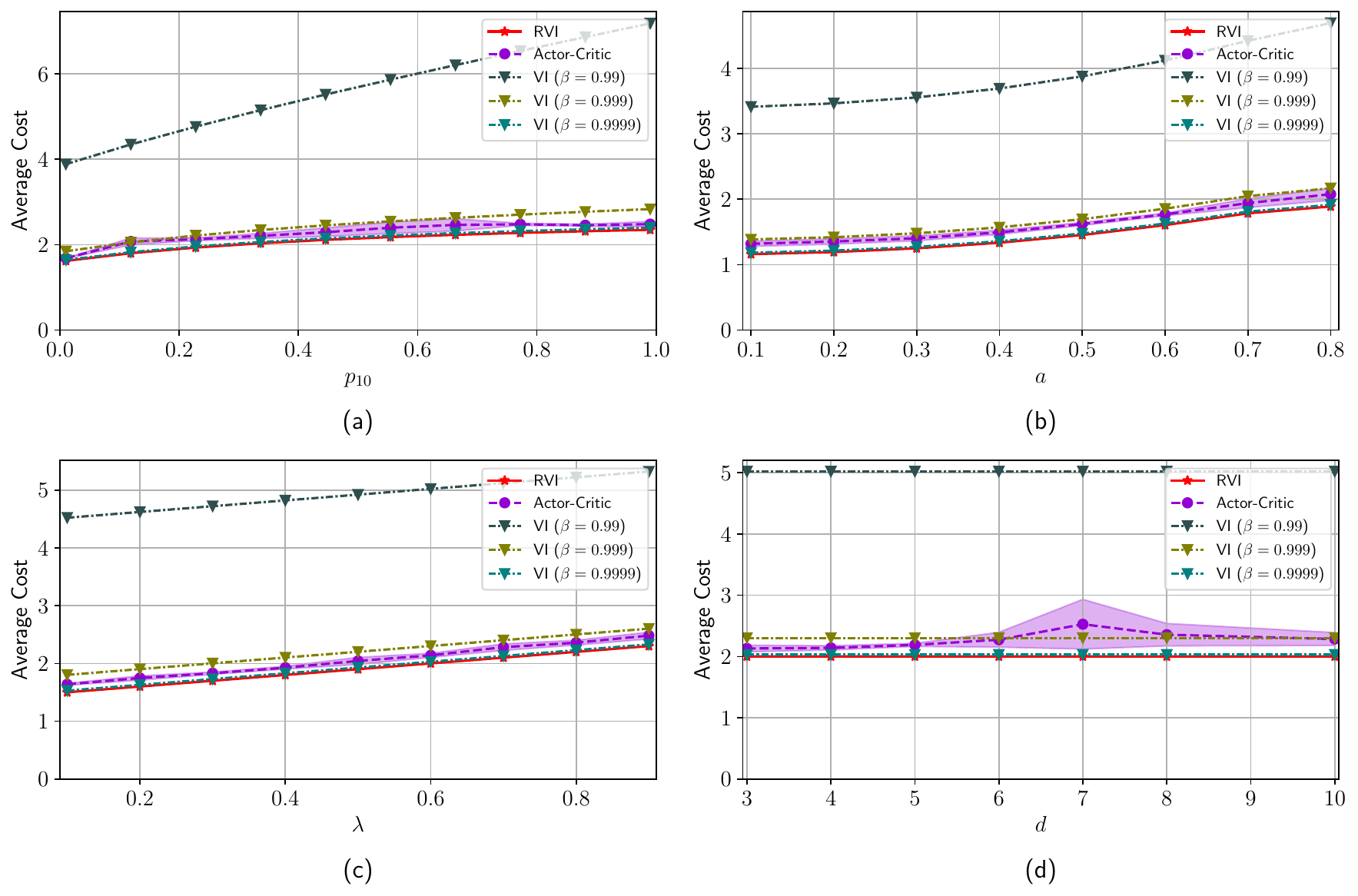}
	\caption{Performance comparison of AC, RVI, and VI under varying system parameters for region $R_1$. AC results are averaged over independent runs, and the mean with $\pm$ one standard deviation is shown. The AC parameters are set as $T = 10^4$, $\alpha_t^{(\omega)} = 1/t^{0.6}$, and $\alpha_t^{(\theta)} = 1/t^{0.7}$. Subplots correspond to: (a) varying $p_{10}$ with $p_{01} = 0.6$, $a = 0.9$, $\lambda = 0.6$, $d = 4$; (b) varying $a$ with $p_{01} = 0.6$, $p_{10} = 0.3$, $\lambda = 0.6$, $d = 4$; (c) varying $\lambda$ with $p_{01} = 0.6$, $p_{10} = 0.3$, $a = 0.9$, $d = 4$; and (d) varying $d$ with $p_{01} = 0.6$, $p_{10} = 0.3$, $a = 0.9$, $\lambda = 0.6$.}
	\label{fig:acrvivir1performancev4}
\end{figure}

\begin{figure}
	\centering
	\includegraphics[width=1\linewidth]{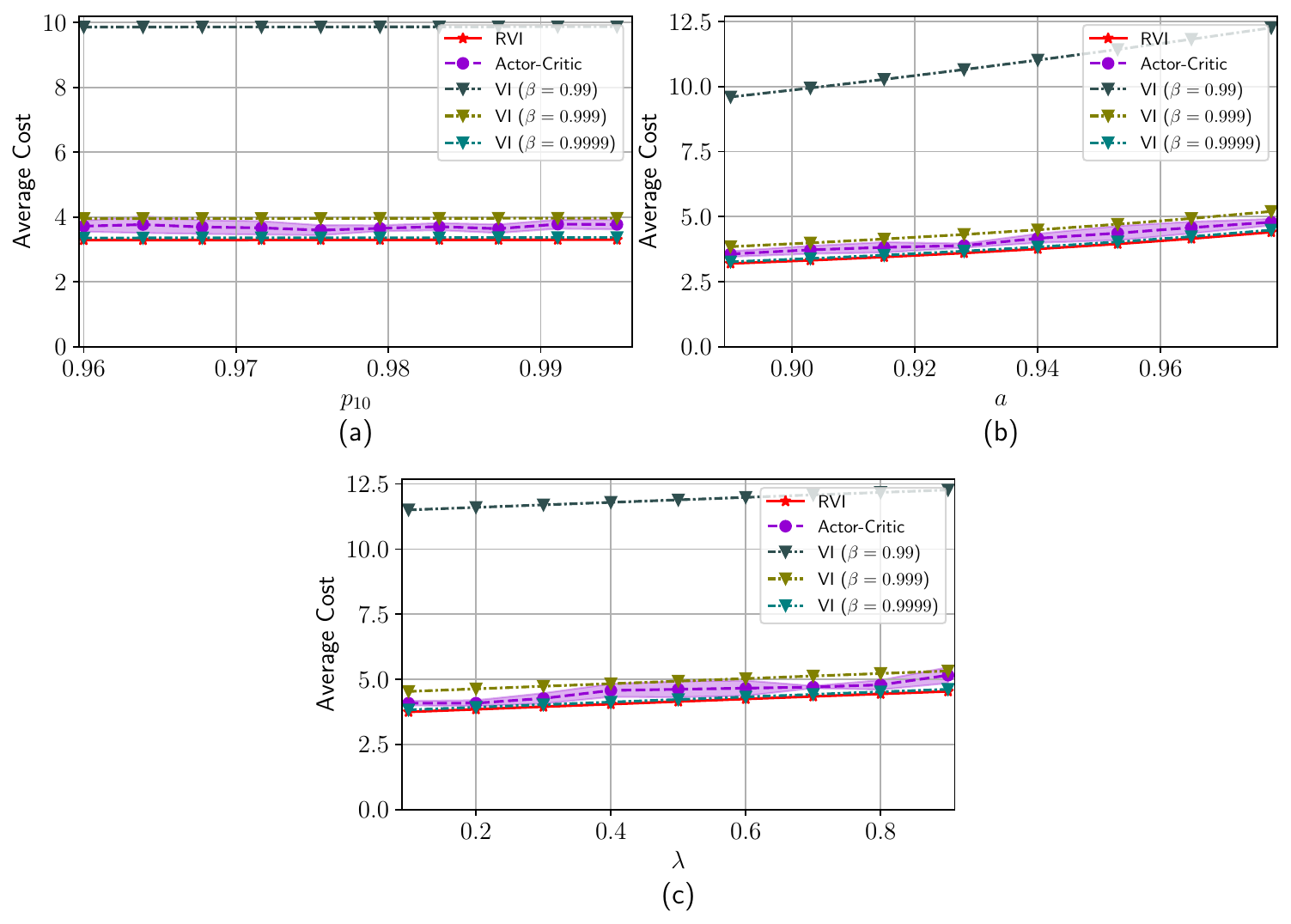}
	\caption{Performance comparison of of AC, RVI, and VI under varying system parameters for region $R_2$. AC results are averaged over independent runs, and the mean with $\pm$ one standard deviation is shown. The AC parameters are set as $T = 10^4$, $\alpha_t^{(\omega)} = 1/t^{0.6}$, and $\alpha_t^{(\theta)} = 1/t^{0.7}$. Subplots correspond to: (a) varying $p_{10}$ $p_{01} = 0.24$, $a = 0.9$, $\lambda = 0.6$, $d = 4$; (b) varying $a$ with $p_{01} = 0.6$, $p_{10} = 0.37$, $\lambda = 0.6$, $d = 4$; and (c) varying $\lambda$ with $p_{01} = 0.6$, $p_{10} = 0.3$, $a = 0.97$, $d = 4$. \textbf{Note}: Unlike the other regions, no subplot varying $d$ is shown here. This is because for fixed $a, p_{01}, p_{10}$, at most one value of $d$ satisfies the conditions defining $R_2$ in Definition~3.}
	\label{fig:acrvivir2performancev4}
\end{figure}

\begin{figure}
	\centering
	\includegraphics[width=1\linewidth]{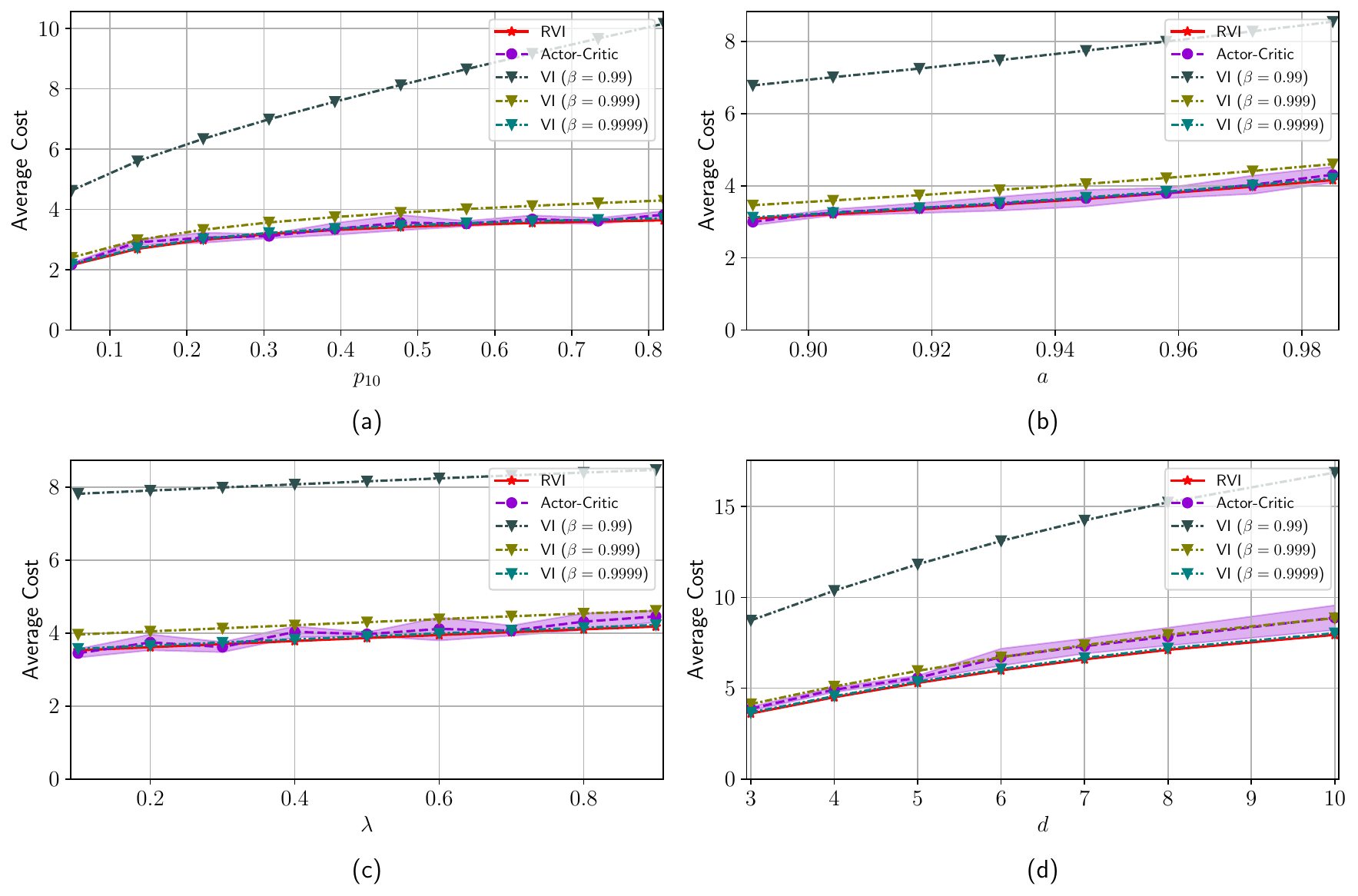}
	\caption{Performance comparison of of AC, RVI, and VI under varying system parameters for region $R_3$. AC results are averaged over independent runs, and the mean with $\pm$ one standard deviation is shown. The AC parameters are set as $T = 10^4$, $\alpha_t^{(\omega)} = 1/t^{0.6}$, and $\alpha_t^{(\theta)} = 1/t^{0.7}$. Subplots correspond to: (a) varying $p_{10}$ $p_{01} = 0.17$, $a = 0.9$, $\lambda = 0.6$, $d = 4$; (b) varying $a$ with $p_{01} = 0.17$, $p_{10} = 0.3$, $\lambda = 0.6$, $d = 4$; (c) varying $\lambda$ with $p_{01} = 0.6$, $p_{10} = 0.3$, $a = 0.97$, $d = 4$; and (d) varying $d$ with $p_{01} = 0.04$, $p_{10} = 0.55$, $a = 0.956$, $\lambda = 0.6$.}
	\label{fig:acrvivir3performancev4}
\end{figure}

\begin{figure}
	\centering
	\includegraphics[width=1\linewidth]{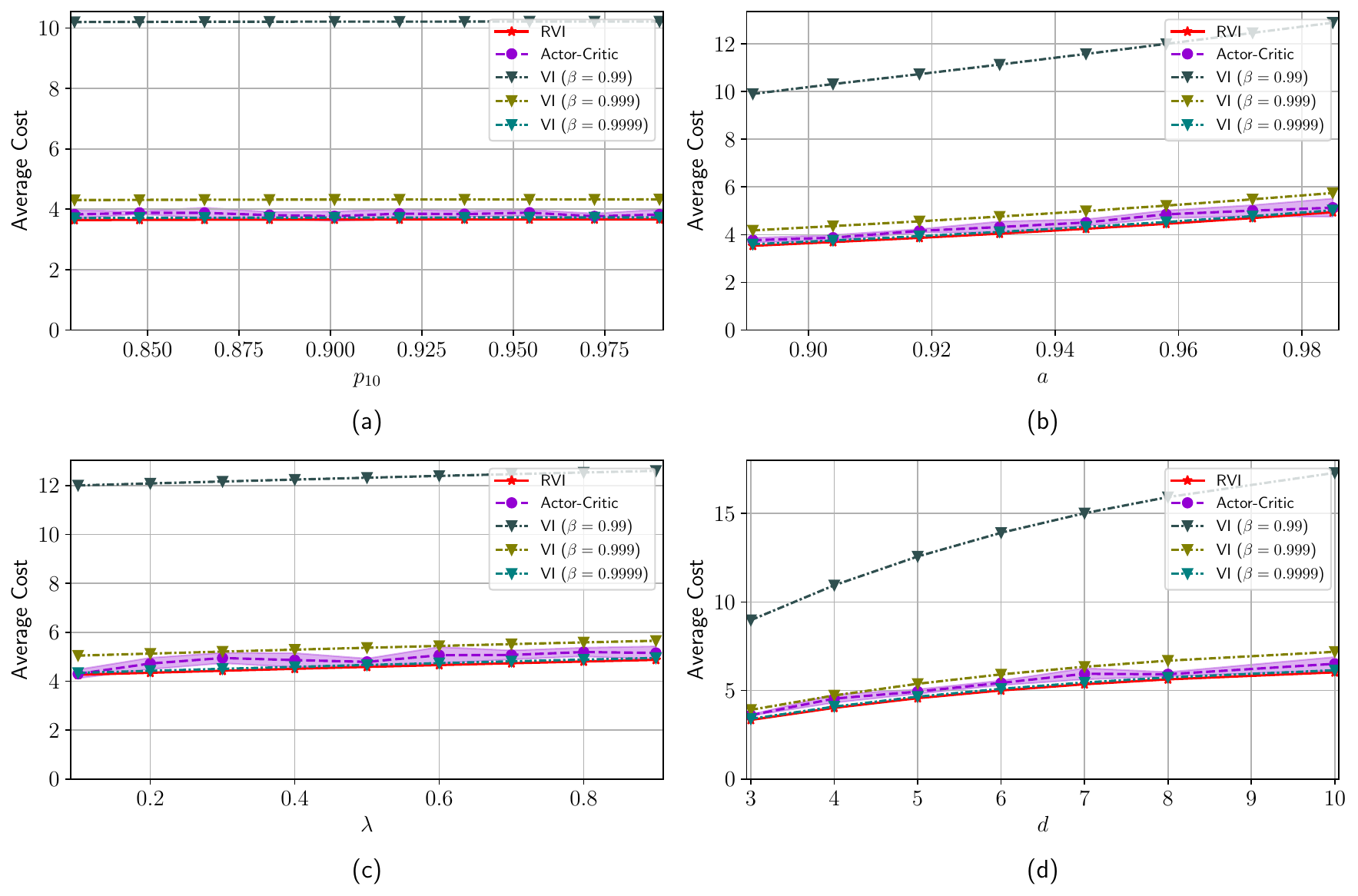}
	\caption{Performance comparison of of AC, RVI, and VI under varying system parameters for region $R_4$. AC results are averaged over independent runs, and the mean with $\pm$ one standard deviation is shown. The AC parameters are set as $T = 10^4$, $\alpha_t^{(\omega)} = 1/t^{0.6}$, and $\alpha_t^{(\theta)} = 1/t^{0.7}$. Subplots correspond to: (a) varying $p_{10}$ with $p_{01} = 0.17$, $a = 0.9$, $\lambda = 0.6$, $d = 4$; (b) varying $a$ with $p_{01} = 0.17$, $p_{10} = 0.83$, $\lambda = 0.6$, $d = 4$; (c) varying $\lambda$ with $p_{01} = 0.17$, $p_{10} = 0.83$, $a = 0.97$, $d = 4$; and (d) varying $d$ with $p_{01} = 0.01$, $p_{10} = 0.96$, $a = 0.922$, $\lambda = 0.6$.}
	\label{fig:acrvivir4performancev4}
\end{figure}

\begin{figure}
	\centering
	\includegraphics[width=1\linewidth]{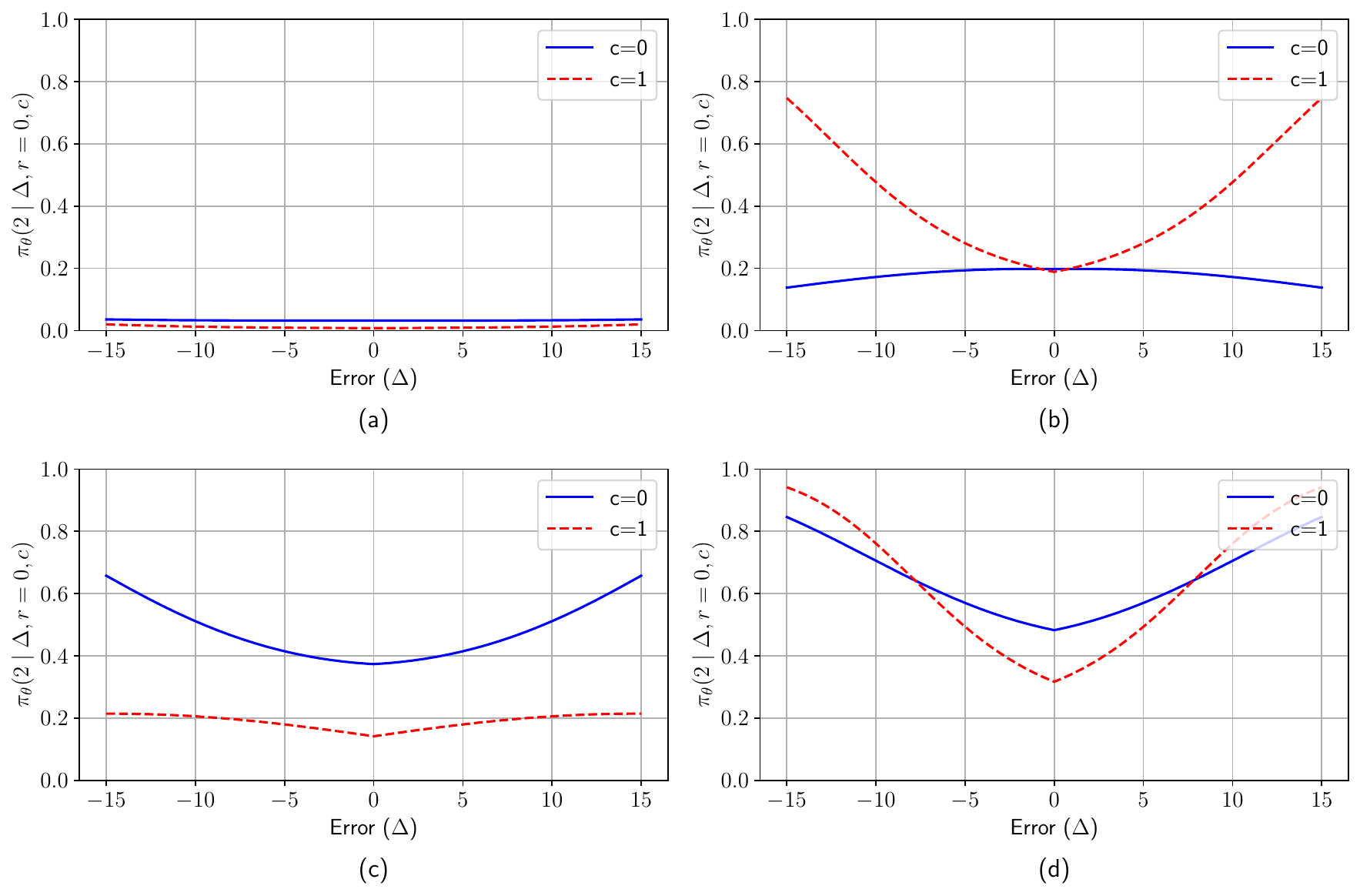}
	\caption{Policy learned by the AC algorithm in regions $R_1,\ldots,R_4$. Each plot shows $\pi_\theta(2\mid\Delta,0,c)$ as a function of $\Delta$ for $c\in\{0,1\}$. The learned policies reflect the threshold-type structure characterized in Theorem~\ref{thm:avg}.}
	\label{fig:softmaxacpolicystructure}
\end{figure}

\section{Conclusion} \label{sec:conclusion}

We study a remote state estimation problem over a hybrid communication architecture consisting of a fast unreliable channel and a slower reliable channel. We formulate the problem as an MDP under an infinite horizon average cost criterion, where the cost comprises squared estimation error and transmission energy consumed. Despite the challenges arising from the continuous state space, unbounded cost, and different channel dynamics, we establish the existence of an optimal stationary policy. We then characterize a threshold-type structure for an optimal policy and show that the
direction of the threshold depends on the AR and channel parameters. These structural results are useful for designing scheduling policies when the system parameters are unknown. We propose an AC learning algorithm that exploits this structure and demonstrate its effectiveness through numerical experiments.
\appendix

\begin{lemma} \label{lemma:even}
    The functions $V\ub(\cdot, r, c)$~\eqref{eq:V}, $Q\ub(\cdot, r, c;u)$~\eqref{eq:Q}, corresponding to the MDP~\eqref{discount_opt_prob} are even, i.e., we have the following for $(\D,r,c) \in \bR \times \{0,1,\ldots,d-1\} \times \{0,1\}$ and $u \in \cU$: 
    \al{
    & V\ub(\D,r,c) = V\ub(|\D|,r,c), \notag\\
    & Q\ub(\D,r,c;u) = Q\ub(|\D|,r,c;u).\label{eq:lemma_even}
    }
    Thus, for any optimal policy $\pi\ub$, we have $\pi\ub(\D,r,c) = \pi\ust(|\D|,r,c)$.
\end{lemma}

\begin{proof}
    We will first prove that~\eqref{eq:lemma_even} holds for VI functions $V\ub_n(\D,r,c)$~\eqref{eq:V_n} and $Q$-functions $Q\ub_n(\D,r,c;u)$~\eqref{eq:Q_n}. We will use induction to show this. The result would then follow from Proposition~\ref{prop:VI}a) and~\cite[Lemma 4.3.8]{Hernandez2012discrete}, since $\lim_{n \rightarrow \infty} V\ub_n(\D,r,c) = V\ub(\D,r,c)$ and $\lim_{n \rightarrow \infty} Q\ub_n(\D,r,c;u) = Q\ub(\D,r,c;u)$. Now we have $V\ub_0(\D,r,c) = 0$~\eqref{eq:V_0}. This is the base case for induction. Next, assume that~\eqref{eq:lemma_even} holds for $k = 1,2,\ldots, n$. We will now show that~\eqref{eq:lemma_even} holds for $k = n+1$. We first consider the case $r = 0, c = 0$. Specifically, we have
    \al{
    & Q\ub_{n+1}(-\D,0,0;1) = \D^2 + \lambda + \beta (1-p_{01}) \notag \\
    & \times \int_{\bR} \eta(\D_+; -a\D, 1) V\ub_n(\D_+, 0, 0) \,d\D_+ \notag \\
    & + \beta p_{01} \int_{\bR} \eta(\D_+;0,1) V\ub_n(\D_+, 0, 1) \,d\D_+ \notag \\
    & = \D^2 + \lambda + \beta (1-p_{01}) \notag \\
    & \times \int_{\bR} \eta(-\D'; -a\D, 1) V\ub_n(-\D', 0, 0) \,d\D' \notag \\
    & + \beta p_{01} \int_{\bR} \eta(-\D';0,1) V\ub_n(-\D', 0, 1) \,d\D' \notag \\
    & = \D^2 + \lambda + \beta (1-p_{01}) \notag \\
    & \times \int_{\bR} \eta(\D'; -a\D, 1) V\ub_n(\D', 0, 0) \,d\D' \notag \\
    & + \beta p_{01} \int_{\bR} \eta(\D';0,1) V\ub_n(\D', 0, 1) \,d\D' \notag \\
    =~& Q\ub_{n+1}(\D, 0, 0;1), \notag
    }
    where the first equality follows from~\eqref{eq:Q_n}. The second equality follows from a change of variable by replacing $\D_+$ with $-\D'$, and the third equality follows from the induction hypothesis that $V\ub_n(\cdot,0,c)$ is even. Hence, $Q\ub_{n+1}(\cdot,0,0;1)$ is even. Similarly, we can show that $Q\ub_{n+1}(\cdot, 0,0;2)$ is also even. Now, upon combining these results for $Q\ub_{n+1}(\cdot,0,0;u)$ with the fact that $V\ub_{n+1}(\cdot,0,0)$~\eqref{eq:V_n} is the pointwise minimum of two even functions $Q\ub_{n+1}(\cdot,0,0;1)$ and $Q\ub_{n+1}(\cdot,0,0;2)$, it follows that $V\ub_{n+1}(\cdot,0,0)$ is even. The cases for $r = 0, c = 1$, and $r \geq 1$ follows on similar lines, and is therefore omitted. This completes the induction. Next, combining~\eqref{eq:lemma_even} with Proposition~\ref{prop:VI} it follows that $\pi\ub(\cdot,r,c)$ are even for each $r \in \{0,1,\ldots,d-1\}$ and $c \in \{0,1\}$. This completes the proof.
\end{proof}

Recall that $\g_1 := \beta a^2 (1-p_{01}), \g_2:= \sum_{i=0}^{d-1} \lf(\beta a^2\rt)^i$.

\begin{lemma} \label{lemma:aux_res}
    Consider the MDP~\eqref{discount_opt_prob} and let Assumption~\ref{ass:sys_chn} hold. Suppose $\sum_{i=0}^{\infty} \g_1^i > \g_2$. Then we have, $\sum_{i=0}^{n-1} \g_1^i + \g_1^n \g_2 > \g_2$ for all $n \in \bN$.
\end{lemma}

\begin{proof}
    We will use induction to prove the result. Let $n=1$. Now, since $\g_1 <1$, we have that $\sum_{i=0}^{\infty} \g_1^i > \g_2$ implies that $1/(1-\g_1) > \g_2$ and thus $1+\g_1\g_2 > \g_2$. This is the base case. Next, assume $\sum_{i=0}^{k-1} \g_1^i + \g_1^k \g_2 > \g_2$ for $k = 1,2,\ldots, n$. We will show that the inequality also holds for $k=n+1$. This follows from the hypothesis and the case for $k=1$ as shown below.
    \nal{
    \sum_{i=0}^{n} \g_1^i + \g_1^{n+1} \g_2 & = 1 + \g_1 \lf(\sum_{i=0}^{n-1} \g_1^i + \g_1^n \g_2\rt) \\
    & > 1 + \g_1\g_2 > \g_2.
    }
    This completes the induction and the proof.
\end{proof}

\begin{lemma} \label{lemma:V_r2_lb}
    Consider the MDP~\eqref{discount_opt_prob} and let Assumption~\ref{ass:sys_chn} hold. The value function $V\ub$~\eqref{eq:V} satisfies the following: 
    \begin{enumerate}
        \item[i)] Suppose $\sum_{i=0}^{\infty} \g_1^i \le \g_2$. Then there exists a constant $K \in \bR_+$ such that for each $\D \in \bR$, $V\ub(\D,0,0) \geq \sum_{i=0}^{\infty} \g_1^i \D^2 + K$.

        \item[ii)] Suppose $\sum_{i=0}^{\infty} \g_1^i > \g_2$. Then there exists a constant $\Tilde{K} \in \bR_+$ such that for each $\D \in \bR$, $V\ub(\D,0,0) \geq \g_2 \D^2 + \Tilde{K}$.
    \end{enumerate}
\end{lemma}

\begin{proof}
    i) We will use induction on VI functions $\{V\ub(\cdot)\}_{n \geq d}$ to prove the result. The proof will then follow from Proposition~\ref{prop:VI} since $\lim_{n \rightarrow \infty} V\ub_n(s) = V\ub(s)$ for each $s \in \cS$. Now, since it follows from Lemma~\ref{lemma:even} that $\{V\ub_n(\cdot)\}_{n \in \bZ_+}$ are even in $\D$, it suffices to show that the result holds for $\D \in \bR_+$. Next, define 
    \nal{
    & L_d = \min\Bigg\{\sum_{i=0}^{d-1} \g_1^i, 1+\g_1+\ldots + \g_1^{d-2}(1+\beta a^2) ,  \\
    & \hspace{3cm}\ldots, 1 + \g_1\sum_{i=0}^{d-2} (\beta a^2)^i \Bigg\},
    }
    and for $n \geq d+1$, $L_n = 1 + \g_1L_{n-1}$. Then using~\eqref{eq:V_0}-\eqref{eq:Q_n} and the fact that $\min\{\alpha_1 x + \alpha_2, \alpha_3 x + \alpha_3\} \geq \min\{\alpha_1, \alpha_3\}x + \min\{\alpha_2, \alpha_4\}$ for any $x \in \bR_+$, and $\alpha_1,\ldots, \alpha_4 \in \bR_+$, it can be shown that $V\ub_d(\D,0,0) \geq \min\lf\{\g_2, 1 + \g_1\sum_{i=0}^{d-2} (\beta a^2)^i, \ldots, \sum_{i=0}^{d-1} \g_1^i\rt\} \D^2 + K_d$ for some constant $K_d \in \bR_+$. Now, since $\sum_{i=0}^{\infty} \g_1^i \leq \g_2$, it follows that $V\ub_d(\D,0,0) \geq L_d \D^2 + K_d$. This is the base case. Next, assume that there exist constants $K_k \in \bR_+$ such that $V\ub_k(\D, 0, 0) \geq L_k \D^2 + K_k$ for $k = d, d+1, \ldots, n$ and $\D \in \bR_+$. We will then show that $V\ub_{n+1}(\D, 0, 0) \geq L_{n+1} \D^2 + K_{n+1}$ for some constant $K_{n+1} \in \bR_+$ and $\D \in \bR_+$. For this, consider the following: 
    \al{
    & Q\ub_{n+1}(\D, 0, 0;1) = \D^2 + \lambda + \beta (1-p_{01}) \notag \\
    & \times \int_{\bR} \eta(\D_+; a\D, 1) V\ub_n(\D_+, 0, 0) \,d\D_+ \notag \\
    & + \beta p_{01} \int_{\bR} \eta(\D_+;0,1) V\ub_n(\D_+, 0, 1) \,d\D_+ \notag \\
    & \geq \D^2 + \beta (1-p_{01})  \int_{\bR} \eta(\D_+; a\D, 1) (L_n \D_+^2 + K_n) \,d\D_+ \notag \\
    & + \beta p_{01} \int_{\bR} \eta(\D_+;0,1) (L_n \D_+^2 + K_n) \,d\D_+ \notag \\
    & = \lf(1 + \g_1L_n\rt) \D^2 + K^{(1)}_{n+1},
    }
    where $K^{(1)}_{n+1} = \beta  (L_n + K_n) + \lambda$. The equality follows from~\eqref{eq:Q_n} and the inequality follows from the hypothesis of $V_n$. Next it follows from~\eqref{eq:Q_n} that
    \al{
    & Q\ub_{n+1}(\D,0,0;2) = \D^2 + \lambda + \beta \sum_{c_1 \in \{0,1\}} p_{0c_1} \notag \\
    & \times \int_{\bR} \eta(\D_+;a\D,1) V\ub_n(\D_1, d-1, c_1) \,d\D_1. \label{eq:lm2_1}
    }
    We will now focus on the term $V_n\ub(\D_1, d-1, c_1)$. Since for $r \geq 1$, the optimal decision is $u = 0$, hence we can iteratively expand $V\ub(\D_1, d-1, c_1)$ as follows: 
    \nal{
    & V\ub_n(\D_1, d-1, c_1) \notag \\
    =~& Q\ub_n(\D_1, d-1, c_1;0) =  \D_1^2\notag \\
    +~& \beta \sum_{c_2} p_{c_1c_2} \int_{\bR} \eta(\D_2; a\D_1, 1) V\ub_{n-1}(\D_2, d-2, c_2) \,d\D_2 \notag \\
    \vdots \notag \\
    =~& \D_1^2 + \beta (1 + (a\D_1)^2) + \ldots \notag \\
    +~& \beta^{d-2} (1 + a^2 + \ldots + (a^2)^{d-3} + (a^2)^{d-2} \D_1^2) \notag \\
    +~& \beta^{d-1} \sum_{c_d} p_{c_1c_d}^{(d-1)} \int_{\bR} \eta(\D_d; 0, \sigma^2_d) V\ub_{n-d + 1}(\D_d, 0, c_d) \,d\D_d, 
    }
    where $p_{cc_+}^{(n)}$ denotes the $n$-step transition probability from state $c$ to $c_+$ of Channel 1. Plugging the above in~\eqref{eq:lm2_1}, we get
    \al{
    & Q\ub_{n+1}(\D,0,0;2) = \sum_{i=0}^{d-1} \lf(\beta a^2\rt)^i \D^2 + K^{(2)}_{n+1},~\label{eq:lm2_2}
    }
    where $K^{(2)}_{n+1}$ is the constant term independent of $\D^2$. Now by  combining~\eqref{eq:lm2_1} and~\eqref{eq:lm2_2} with the fact that $\min\{\alpha_1 x + \alpha_2, \alpha_3 x + \alpha_3\} \geq \min\{\alpha_1, \alpha_3\}x + \min\{\alpha_2, \alpha_4\}$ for any $x \in \bR_+$ and $\alpha_1,\ldots, \alpha_4 \in \bR_+$, we have $V\ub_{n+1}(\D,0,0) \geq \min\{L_{n+1}, \g_2\} \D^2 + K_{n+1}$, where $K_{n+1} = \min\{K_{n+1}^{(1)},K_{n+1}^{(2)}\}$. Since $\sum_{i=0}^{\infty} \g_1^i \leq \g_2$, it implies that $V\ub_{n+1}(\D,0,0) \geq L_{n+1}\D^2 + K_{n+1}$. This completes the induction. The proof is then completed by taking limit on both sides of the former inequality as $n \rightarrow \infty$. Because $\g_1 < 1$ from Assumption~\ref{ass:sys_chn}, it follows that $\lim_{n \rightarrow \infty} L_n = \sum_{i=0}^{\infty} \g_1^i$, and there exists some constant $K$ such that $\lim_{n \rightarrow \infty} K_{n+1} = K$. 
    
    2) We will use induction on VI functions $\{V\ub(\cdot)\}_{n \geq d}$ to prove the result. It follows from Lemma~\ref{lemma:even} that it suffices to consider $\D \in \bR_+$. Now define $L'_d = \min\lf\{\g_2, 1 + \g_1\sum_{i=0}^{d-2} (\beta a^2)^i, \ldots, \sum_{i=0}^{d-1} \g_1^i\rt\}$, and for $n \geq d+1$, $L'_n = 1 + \g_1L'_{n-1}$. Then using~\eqref{eq:V_0}-\eqref{eq:Q_n} it can be shown that $V\ub_d(\D,0,0) \ge L'_d \D^2 + K'_d$ for some constant $K'_d \in \bR_+$. This is the base case. Now assume that there exist constants $K'_k$ such that $V\ub_k(\D,0,0) \ge L'_k \D^2 + K'_k$ for $k = d, d+1, \ldots, n$. We will show that $V\ub_{n+1}(\D,0,0)$ also satisfies this inequality. The proof for this follows on similar lines as part 1) and is therefore omitted. We will now focus on the coefficient of $\D^2$. The proof is then completed by taking the limit as $n \rightarrow \infty$ on both side of $V\ub_n(\D,0,0) \ge L'_n \D^2 + K'_n$ and the following observation
    \nal{
    & \lim_{n \rightarrow \infty} L'_n = \lim_{n \rightarrow \infty} \min\Bigg\{\g_2, 1+\g_1\g_2, \ldots, 1+\g_1^{n-d}\g_2, \\
    & 1+\g_1+\ldots + \g_1^{n-d+1} \sum_{i=0}^{d-2} \lf(\beta a^2\rt)^i, \ldots, 1+\g_1 \\
    & +\ldots+\sum_{i=0}^{n-2}(1+\beta a^2), \sum_{i=0}^{n-1} \g_1^i \Biggr\} \\
    & = \lim_{n \rightarrow \infty} \min\Bigg\{\g_2, 1+\g_1+\ldots + \g_1^{n-d+1} \sum_{i=0}^{d-2} \lf(\beta a^2\rt)^i\\
    & , \ldots, 1+\g_1+\ldots+\sum_{i=0}^{n-2}(1+\beta a^2), \sum_{i=0}^{n-1} \g_1^i\Bigg\} \\
    & = \min\Bigg\{\g_2, \sum_{i=0}^{\infty} \g_1^i\Bigg\} = \g_2,
    }
    where the second equality follows from Lemma~\ref{lemma:aux_res} and the third equality follows since we have that $\g_1 <1$.
\end{proof}

\begin{lemma} \label{lemma:inc_V}
    The value function $V\ub$~\eqref{eq:V} corresponding to the MDP~\eqref{discount_opt_prob} satisfies the following property: for any $\D, \D' \in \bR$ such that $|\D| \geq |\D'|$, we have that $V\ub(\D,r,c) \geq V\ub(\D',r,c)$ for each $r \in \{0,1,\ldots d-1\}$ and $c \in \{0,1\}$.
\end{lemma}

\begin{proof}
    We use induction on the value iteration functions $\{V\ub_n\}_{n \in \bZ_+}$. It follows from Lemma~\ref{lemma:even} that $V\ub_n(\cdot, r, c)$ is an even function for all $n$. Thus, it suffices to show that for any $\D \geq \D' \geq 0$, we have $V\ub_n(\D, r, c) \geq V\ub_n(\D', r, c)$ for all $r \in \{0,1,\ldots d-1\}$ and $c \in \{0,1\}$.

    For $n=0$, $V\ub_0(\D, r, c) = 0$ for all states~\eqref{eq:V_0}. Hence, the base case is true. Now, we assume that $V\ub_k(\D, r, c) \geq V\ub_k(\D', r, c)$ holds for $k = 1, \ldots, n$ and for any $\D \geq \D' \geq 0$. We will show the claim holds for $k = n+1$.

    We first consider the case $r = 0$. It follows from~\eqref{eq:Q_n} that for $u=1$ and $c \in \{0,1\}$ we have,
    \al{\label{eq:lemma_inc}
    & Q\ub_{n+1}(\D,0,c;1) \notag \\
    & = \D^2 + \lambda + \beta p_{c0} \int_{\bR} \eta(\D_+;a\D,1) V\ub_n(\D_+,0,0) \,d\D_+ \notag \\
    & \quad + \beta p_{c1} \int_{\bR} \eta(\D_+;0,1) V\ub_n(\D_+, 0, 1) \,d\D_+.
    }
    Since $\D \geq \D' \geq 0$, it immediately follows that $\D^2 \geq (\D')^2$. Next, for the first integral term in the right hand side of the above equation, by the induction hypothesis and Lemma~\ref{lemma:even}, we have that $V\ub_n(\cdot,0,0)$ is an even function that is monotonically non-decreasing on $\bR_+$. Then, it follows from~\cite[Lemma 4.7.2]{puterman2014markov} that,
    \nal{
    & \int_{\bR} \eta(\D_+;a\D,1) V\ub_n(\D_+,0,0) \,d\D_+ \notag \\
    & \geq \int_{\bR} \eta(\D_+;a\D',1) V\ub_n(\D_+,0,0) \,d\D_+.
    }
    Combining the above inequality with~\eqref{eq:lemma_inc} yields $Q\ub_{n+1}(\D,0,c;1) \geq Q\ub_{n+1}(\D',0,c;1)$. We can similarly show that $Q\ub(\D,0,c;2) \geq Q\ub(\D',0,c;2)$. Because $V\ub_{n+1}(\D,0,c) = \min_{u \in \{1,2\}} Q\ub_{n+1}(\D,0,c;u)$, and the pointwise minimum of two non-decreasing functions is non-decreasing, we conclude $V\ub_{n+1}(\D,0,c) \geq V\ub_{n+1}(\D',0,c)$. 
    
    The proof for $r \in \{1,2,\ldots, d-1\}$ and $c \in \{0,1\}$ follows on similar lines and is therefore omitted. This completes the induction.

    Finally, because $\lim_{n \to \infty} V\ub_n(s) = V\ub(s)$ by Proposition~\ref{prop:VI}, the non-decreasing property is preserved in the limit. Hence, $V\ub(\D,r,c) \geq V\ub(\D',r,c)$ for all $|\D| \geq |\D'|$. This completes the proof.
\end{proof} 

\begin{definition}[Non-trivial measure]
    Consider a measurable space $(\cX, \cF)$. A measure $\mu:\cF \rightarrow [0,\infty]$ is said to be non-trivial if there exists at least one measurable set $X \in \cF$ for which $\mu(X) > 0$ holds.
\end{definition}

\begin{definition}[$\vp$-Irreducibility]
    A Markov chain $\mathbf{\Phi} = \{\Phi(t)\}$ with state space $\cX$ is called $\vp$-irreducible if there exists a measure $\vp$ on $\bB(\cX)$ such that for all $X \in \bB(\cX)$ with $\varphi(X)>0$, the following holds: 
    \[
    \mathbb \bP(\tau_X < \infty | \Phi(0) = s) > 0,
    \qquad \text{for all } s \in \cX,
    \]
where $\tau_X := \inf\{t \ge 0:\Phi(t)\in X\}$.
\end{definition}

\begin{definition}[$\psi$-Irreducibility]
    A Markov chain $\mathbf{\Phi} = \{\Phi(t)\}$ is $\psi$-irreducible if it is $\vp$-irreducible for some $\vp$ and the measure $\psi$ is a maximum irreducibility measure that satisfies the conditions of~\cite[Proposition 4.2.2]{meyn2012markov}.
\end{definition}

We define $\cA:= \{(\D,r,c): \D \in X \subset \bR, r=0, c=1\}$ for some Borel set $X \subset \bR$. Let $\pi^{(hyb.)}$ denote a hybrid policy that chooses Channel 1 (i.e., $u(t)=1$) at each time step $t$, $t \in \bZ_+$ until the Markov process $\{(\D(t),r(t),c(t-1))\}$ hits the set $\cA$, and then switches to an optimal policy for the MDP~\eqref{discount_opt_prob}. Under the policy $\pi^{(hyb.)}$, the Markov process $\{(\D(t), r(t), c(t-1))\}$ forms a time-homogeneous Markov chain. 

\begin{lemma} \label{lemma:irreducible}
    Let Assumption~\ref{ass:sys_chn} hold. Then, the Markov process $\{(\D(t),r(t), c(t-1))\}$, induced by policy $\pi^{(hyb.)}$ is $\psi$-irreducible for some non-trivial and non-negative measure $\psi$.
\end{lemma}

\begin{proof}
     Let $F_w$ be the probability measure on $(\mathbb R,\mathcal \bB(\bR))$ induced by the Gaussian noise process $\{w(t)\}$~(cf.~\eqref{source}), i.e., for any Borel set $X \subset \bR$
    \nal{
    F_w(X) = \int_X \eta(\D;0,1) \,dD.
    }
    We define the
non-trivial measure $\vp:=F_w\otimes\delta_0\otimes\delta_1$ on $\bB(\cS)$. We will prove that for all $s\in\cS$, $S\in\bB(\cS)$ with $\vp(S)>0$ we have that
\al{
\bP \big(\tau_S<\infty\mid s(0)=s\big)>0,
\label{eq:irr-goal}
}
where $\tau_S$ is the first time the Markov process hits the set $S$.
Fix such an $S$ and set $X:=\{\D \in \bR:(\delta,0,1)\in S\}$. Then
$\vp(S)=F_w(X)>0$.

We use the following two facts.

\begin{itemize}
\item[(F1)] Whenever $r(t)\ge1$, we have that $u(t)=0$, and at the step with $r(t)=1$ the packet is delivered at the next time step, giving
$r(t+1)=0$ and, by Table~\ref{tab:df}, $\D(t+1)\sim\cN(0,\sigma^2_d)$
independently of the past error and of the channel process.

\item[(F2)] Whenever $r(t)=0$, the policy applies some $u(t)\in\{1,2\}$. Under
$u(t)=1$, Table~\ref{tab:df} gives $c(t)=1$, $\D(t+1)\sim\cN(0,1)$ and
$r(t+1)=0$ with probability $p_{c(t-1),1}$. Under $u(t)=2$, $r(t+1)=d-1$.
\end{itemize}

\textbf{Step 1(reaching $\{r=0\}$).} If $r(0)=0$, we set $t_0 =0$. If $r(0)=r\ge1$,
then by (F1) $u \equiv -$ for $r$ steps and the process reaches a state
$(\D_0,0,c_0)$ at time $t_0 =r\le d-1$. In either case the process is at some
state $(\D_0,0,c_0)$ with $r(t_0)=0$ at a time $t_0\le d-1$, with probability one.

\textbf{Step 2 (reaching $S$ from $(\D_0,0,c_0)$).} Let $u_0\in\{1,2\}$ be the
action taken at time $t_0$.

If $u_0=1$, then by (F2), with probability at least $p_{c_0,1}$ we have
$c(t_0)=1$, $r(t_0+1)=0$ and $\D(t_0+1)\sim\cN(0,1)$. Now since $\{\D(t_0+1)\in X\}$, it implies $s(t_0+1)\in S$ that
\nal{
\bP\big(s(t_0+1)\in S\mid s(t_0)=(\D_0,0,c_0)\big)\ge p_{c_0,1}\,F_w(X)>0.
}

If $u_0=2$, then $r(t_0+1)=d-1\ge1$, so by (F1) the packet
is delivered at time $t_0+d$, giving $r(t_0+d)=0$ and
$\D(t_0+d)\sim\cN(0,\sigma^2_d)$ independently of the channel. The channel
state is then $c(t_0+d-1)$, with
$\bP\big(c(t_0+d-1)=1\mid c(t_0-1)=c_0\big)=p^{(d)}_{c_0,1}$. Since
$\{\D(t_0+d)\in X,\ c(t_0+d-1)=1\}$ implies $s(t_0+d)\in S$, independence of the
noise and channel processes yields
\nal{
\bP\big(s(t_0+d)\in S\mid s(t_0)=(\D_0,0,c_0)\big)\ge F_{\epsilon}(X)\,p^{(d)}_{c_0,1}>0.
}

In both cases the probability of reaching $S$ within $d$ further steps is strictly
positive. Combined with Step~1, this gives
$\bP(\tau_S\le t_0+d\mid s(0)=s)>0$, which
proves~\eqref{eq:irr-goal}. Hence the process is $\vp$-irreducible, and
by~\cite[Proposition~4.2.2]{meyn2012markov} there exists a maximal irreducibility
measure $\psi$ (with $\vp$ absolutely continuous with respect to $\psi$) relative
to which the process is $\psi$-irreducible. This completes the proof.
\end{proof}

\begin{definition}[Small sets]
    Consider a Markov chain $\mathbf{\Phi} = \{\Phi(t)\}$ with state space $\cX$ and transition probability kernel $P$. A set $X \in \bB(\cX)$ is called a small set for the Markov chain $\mathbf{\Phi}$ if there exists a $t_0 > 0$, and a non-trivial measure $\nu_{t_0}$ on $\bB(\cX)$, such that for all $s \in X$, $B \in \bB(\cX)$, $P^{t_0}(s,B) \geq \nu_{t_0}(B)$, where $P^t$ is the $t$-step transition probability kernel. We call such a set $X$ as $\nu_{t_0}$-small.
\end{definition}

\begin{definition}[Atoms]
    A set $X \in \bB(\cX)$ is called an atom for a Markov chain $\mathbf{\Phi} = \{\Phi(t)\}$ with state space $\cX$ and transition probability kernel $P$ if there exists a measure $\nu$ on $\bB(\cX)$ such that $P(s,B) = \nu(B)$, for $s \in X$.
\end{definition}

Let $(\D', r', c') = (0, 0, 1)$ be a reference state and $h^{\beta}(s) := V\ub(s) - V\ub(\D', r', c')$, $s \in \cS$.

\begin{lemma} \label{lemma:cond_avg_cost}
    Consider the MDP~\eqref{discount_opt_prob} and let Assumption~\ref{ass:sys_chn} hold. Then the value function $V\ub$~\eqref{eq:V} corresponding to~\eqref{discount_opt_prob} satisfies the following 

    \begin{enumerate}
        \item[i)] There exists a constant $C \geq 0$ such that $(1 - \beta) V\ub(\D', r', c') \leq C$ for all $\beta \in (0,1)$.

        \item[ii)] There exists a constant $N \geq 0$ such that $N \leq h\ub(s)$ for all $s \in \cS$ and $\beta \in (0,1)$.

        \item[iii)] There exists a non-negative function $\kappa: \cS \rightarrow \bR_+$ such that $h\ub(s) \leq \kappa(s)$ for all $s \in \cS$ and $\beta \in (0,1)$.
    \end{enumerate}
\end{lemma}

\begin{proof} 
    i) It follows from~\eqref{j_beta} and~\eqref{eq:upper_bound} that $(1 - \beta) V\ub(\D', r', c') \le (1 - \beta) J\ub(\D', r', c') \leq 1/((1-p_{01}) (1 - a^2(1-p_{01}))) + \lambda$.

    ii) It follows from Lemma~\ref{lemma:inc_V} that $h\ub(s) \geq 0$ since $V\ub(s) \geq V\ub(\D',r',c')$. Hence, $N = 0$.

    iii) Fix $\varepsilon > 0$ and define 
    \nal{
    \cA(\varepsilon) := \{(\D,r,c) \in \cS: |\D| \le \varepsilon, r = 0, c = 1\}.
    }
    We consider the hybrid policy $\pi^{(hyb.)}$ that transmits via Channel 1 at each time $t \in \bZ_+$ until the Markov process $(\D(t),r(t),c(t-1))$ first hits the set $\cA(\epsilon)$ and then switches to an optimal policy for the MDP~\eqref{discount_opt_prob}. Note that the $\cA(\ve)$ is the set $\cA$ by taking $X = [\ve,\ve]$, and hence we continue using $\pi^{(hyb.)}$ to denote the hybrid policy. Let $\tau_{\cA(\varepsilon)} = \inf\{t \geq 0: (\D(t),r(t),c(t-1)) \in \cA(\varepsilon)\}$ be the first hitting time of $\cA(\varepsilon)$. Then for any initial state $s(0) = (\D,r,c) \in \cS$ we have from~\eqref{eq:V}
    \al{
    & V\ub(\D,r,c) \leq J\ub(\D,r,c; \pi^{(hyb.)}) \notag \\
    & = \bE_{\pi^{(hyb.)}}\Biggl[\sum_{t=0}^{\tau_{\cA(\ve)} - 1} \beta^t \ell(\D(t),r(t),c(t-1)) \bigg|  \notag \\
    & \hspace{5.5cm} s(0)= (\D,r,c)\Biggr] \notag \\
    & + \bE_{\pi^{(hyb.)}}\Biggl[\sum_{t= \tau_{\cA(\ve)}}^{\infty} \beta^t \ell(\D(t),r(t),c(t-1)) \bigg|  \notag \\
    & \hspace{5cm} s(0) = (\D,r,c)\Biggr].~\label{eq:split}
    }
    We will now focus on the second term of the above inequality. It follows from Lemma~\ref{lemma:irreducible} that the Markov process $(\D(t),r(t),c(t-1))$ induced by $\pi^{(hyb.)}$ is $\psi$-irreducible. It then follows from Theorems 5.2.2 and~5.2.3 of~\cite{meyn2012markov} that the sets $\cA(\ve)$ are small sets that satisfy the minorization condition~\cite[p. 105]{meyn2012markov}. Thus we can apply the Nummelin splitting construction~\cite{nummelin1978splitting},~\cite[p. 106]{meyn2012markov} to split the Markov process chain $(\D(t),r(t),c(t-1))$ induced by the policy $\pi^{(hyb.)}$. The Nummelin splitting technique allows us to construct an atom arbitrarily close to $(0,0,1)$ for the split chain. Since all transitions out of an atom are identical~\cite[p. 107]{meyn2012markov}, it follows that once the Markov process $(\D(t),r(t),c(t-1))$ hits $\cA(\ve)$ the evolution of the process is same as the process starting from an initial state $(0,0,1)$. Thus we have that
    \nal{
    & \bE_{\pi^{(hyb.)}}\Biggl[\sum_{t= \tau_{\cA(\ve)}}^{\infty} \beta^t \ell(\D(t),r(t),c(t-1)) \bigg|  \\
    & \hspace{5.5cm} s(0) = (\D,r,c)\Biggr] \\
    & = \bE_{\pi^{(hyb.)}} \lf[\beta^{\tau_{\cA(\ve)}}\rt] V\ub(0,0,1) \\
    & \leq V\ub(0,0,1).
    }
    Combining the above with~\eqref{eq:split} we have 
    \al{
    & V\ub(s) - V\ub(0,0,1) \notag \\
    & \leq \bE_{\pi^{(hyb.)}}\Biggl[\sum_{t=0}^{\tau_{\cA(\ve)} - 1} \beta^t \ell(\D(t),r(t),c(t-1)) \bigg|  \notag \\
    & \hspace{5.5cm} s(0) = (\D,r,c)\Biggr] \notag \\
    & \leq \bE_{\pi^{(hyb.)}}\Biggl[\sum_{t=0}^{\tau_{\cA(\ve)} - 1} \ell(\D(t),r(t),c(t-1)) \bigg| \notag \\
    & \hspace{5cm} s(0)= (\D,r,c)\Biggr]. \label{eq:end}
    }
    This completes the proof with $\kappa(\D,r,c)$ equal to the right hand side of~\eqref{eq:end}. The finiteness of $\kappa(\D,r,c)$ is shown in Lemma~\ref{lemma:kappa_finite}.
\end{proof}

For any admissible policy $\pi$, we denote $\zeta(s, \pi ) := \limsup\limits_{T \rightarrow \infty} \frac{1}{T} \bE_{\pi} \lf[\sum_{t=0}^{T} \lf(\D^2 + \lambda \mathds{1}\lf(u(t) \in \{1,2\} \rt) \rt)\rt]$. 

\begin{lemma} \label{lemma:GA}
    Consider the average cost problem~\eqref{opt_prob}. Then for any admissible policy $\pi$ we have that 
    \al{
    \inf_{s \in \cS} \inf_{\pi} \zeta(s, \pi) < \infty. \label{eq:GA}
    }
\end{lemma}

\begin{proof}
    We will first show that~\eqref{eq:GA} is equivalent to showing that there exists a state $\Tilde{s} \in \cS$ and a policy $\Tilde{\pi}$ such that $\zeta(\Tilde{s}, \Tilde{\pi}) < \infty$. For this it suffices to show, i) if $\zeta(\Tilde{s}, \Tilde{\pi}) < \infty$, then it  implies that~\eqref{eq:GA} holds, and ii) conversely, if \eqref{eq:GA} holds, then there exists $\Tilde{s}$ and $\Tilde{\pi}$ such that $\zeta(\Tilde{s}, \Tilde{\pi}) < \infty$. We will begin with i).

    i) follows immediately since $\inf_{s \in \cS} \inf_{\pi} \zeta(s,\pi) \leq \zeta(\Tilde{s}, \Tilde{\pi}) < \infty$. Next we will show ii) holds.

    ii) By the definition of infimum, for any $\epsilon > 0$, there exists a state $\Tilde{s} \in \cS$ and an admissible policy $\Tilde{\pi}$ such that $\zeta(\Tilde{s},\Tilde{\pi}) < \inf_{s \in \cS} \inf_{\pi} \zeta(s, \pi) + \epsilon$. Since the infimum is assumed to be finite, the right hand side of the former expression is finite, and hence $\zeta(\Tilde{s},\Tilde{\pi})$ is finite. Thus ii) follows by choosing any $\epsilon > 0$.

    The proof then follows from Condition 5.4.5 (a) of~\cite{Hernandez2012discrete} and~\cite[Theorem 5.4.6]{Hernandez2012discrete} once we verify that Assumptions 4.2.1 and 5.4.1 hold in our case. In Lemmas~\ref{lemma:VI_assump} and~\ref{lemma:cond_avg_cost}, we verify Assumptions 4.2.1 and 5.4.1, respectively, are satisfied in our case. This completes the proof. 
\end{proof}

\begin{lemma} \label{lemma:kappa_finite}
    Consider the Markov process $\{(\D(t),r(t), c(t-1))\}$ induced by the policy $\pi^{(hyb.)}$.  Let $s(0)$ denote the initial state of the process. For any initial state $s(0) = s$, we have,
    \nal{
    \kappa(s) & = \bE_{\pi^{(hyb.)}}\Biggl[\sum_{t=0}^{\tau_{\cA(\ve)} - 1} \ell(\D(t),r(t),c(t-1)) \bigg|  s(0)= s\Biggr], \\
    & < \infty,
    }
    where $\kappa(s)$ is as defined in Lemma~\ref{lemma:cond_avg_cost}-iii).
\end{lemma}

\begin{proof}
    We first consider an initial state $s = (\D_0,r_0,c_0)$ with $r_0 \ge 1$. Under the policy $\pih$, the action taken is $u(t) = 0$ for $t = 0, 1, \ldots, r_0-1$. Hence, the error evolves as $\D(t) = a^t \D_0 + \sum_{k=0}^{t-1} a^{t-1-k} w(k)$. Thus, for $t = 0, 1, \ldots, r_0-1$ we have,
    \nal{
    \bE_{\pih}\lf[\D(t)^2\rt] = a^{2t} \D^2 + \sum_{k=0}^{t-1} a^{2k} < \infty,
    }
    and hence,
    \al{
    \bE_{\pih}\lf[\sum_{t=0}^{r_0-1} \D(t)^2\rt] < \infty. \label{eq:r0}
    }
    Next, we note that $\ta \ge r_0$. This is because hitting the set $\cA(\ve)$ requires $r = 0$, which first occurs at time step $r_0$. Specifically, at $t = r_0$, the state of the Markov process $\{\D(t),r(t),c(t-1))\}$ is $(\D(r_0), 0, c(r_0 -1))$ with $\D(r_0) \sim \cN(0, \sigma^2_d)$. Then, it follows from~\eqref{eq:r0} that it suffices to prove the proposition for an initial state with $r_0 = 0$. Hence, for the subsequent proof we fix the initial state as $s = (\D_0,0,c_0)$. Without loss of generality, we assume that $s \notin \cA(\ve)$ (otherwise $\tau_{\mathcal{A}(\varepsilon)} = 0$ and $\kappa(s) = 0$ trivially). Moreover, under $\pih$, we have that $u(t) \equiv 1$ and $r(t) \equiv 0$ for all $t < \ta$. Now, we make the following important observation.

    \emph{Observation 1}: Every state with $(r=0,c=1)$ reached under policy $\pih$ at time step $t \ge 1$ occurs immediately after a successful transmission over Channel 1. Because under $\pih$ we have $u(t-1) = 1$, so $c(t-1) = 1$ implies that the packet is successfully delivered to the estimator at time step $t-1$. Hence, $\D(t) = w(t-1) \sim \cN(0,1)$. As a result, the Markov process $\{\D(t),r(t),c(t-1)\}$ can hit the set $\cA(\ve)$ at time step $t$ only after a successful transmission at $t-1$ and thus, $\D(t) \sim \cN(0,1)$.

    Next, let $\varsigma_1 := \inf\{t \ge 0 : c(t) = 1\}$ denote the time of the first successful transmission over Channel 1, and $\varsigma_{j+1} := \inf\{t > \varsigma_j : c(t) = 1\}$ for $j \ge 1$. Since, Channel 1 is a two state Markov chain with $p_{01}, p_{10} \in (0,1)$, we have that $\varsigma_j < \infty$ almost surely for all $j \in \bN$. We denote $w_j:= w(\varsigma_j)$. Then, $\D(\varsigma_j + 1) = w_j$. Moreover, since each $\varsigma_j$ is measurable with respect to $\sigma(\{c(t)\})$ which is independent of the noise process $\{w(t)\}$, it follows that $w_j \sim \cN(0,1)$ for all $j \ge 1$. 

    Now, we note that each $\varsigma_{j} + 1$ is a stopping time for the Markov process $\{\D(t),r(t),c(t-1)\}$, and the state at time $\varsigma_j + 1$ is $(w_j, 0, 1)$. Hence, by the strong Markov property~\cite{meyn2012markov}, given the information till $\varsigma_{j}+1$, the evolution of the Markov process $\{\D(t),r(t),c(t-1)\}$ depends only on $(w_j,0,1)$ and is independent of the history. It then follows that for $E_j := \{\varsigma_j+1, \varsigma_j + 2, \ldots, \varsigma_{j+1}\}$, the following squared cumulative error forms an i.i.d. sequence.
    \nal{
    S_j := \sum_{t \in E_j} \D(t)^2.
    }

   Next, we note that based on Observation 1, the Markov chain enters $\cA(\ve)$ at time step $\varsigma_j + 1$ if and only if $|\D(\varsigma_j + 1)| \le \ve$ which is equivalent to $|w_j| \le \ve$. Now, define
    \nal{
    J_{\ve}:= \inf\{j \ge 1: |w_j| \le \ve\},
    }
    i.e., $J_{\ve}$ is the index of the first success whose value is less than $\ve$. Then, the time at which the Markov process hits $\cA(\ve)$ corresponds to the $J_{\ve}$-th success, i.e., $\ta = \varsigma_{J_{\ve}} + 1$. Since $w_j$ are i.i.d. with $q:=\bP(|w_j| \le \ve)$, we have $J_{\ve} \sim \text{Geom}(q)$, $\bE[J_{\ve}] = 1/q < \infty$. 

    Now, we decompose the $\kappa(s)$ term into three parts and show that each part is finite. For this, denote $I_0 = \{0, 1, \ldots, \varsigma_1\}$ as the initial interval before the first successful transmission, and let $S_0 := \sum_{t \in I_0} \D(t)^2$. Then, we can partition the time $t < \ta$ as $I_0 \cup E_1 \cup \ldots E_{J_{\ve}-1} = \{0,1,\ldots, \varsigma_{J_{\ve}}\}$. This implies that we can write $\kappa(s)$ as follows,
    \al{
    & \kappa(s) = \bE_{\pih} \lf[\sum_{t=0}^{\ta -1} \lambda \rt] + \bE_{\pih}\lf[\sum_{t =0 }^{\ta - 1} \D(t)^2\rt] \notag \\
    & = \bE_{\pih} \lf[\sum_{t=0}^{\ta -1} \lambda \rt] + \bE_{\pih}\lf[\sum_{t = 0}^{\varsigma_{J_{\ve}}} \D(t)^2 \Big | s(0) = s\rt] \notag \\
    & = \lambda \bE_{\pih}\lf[\ta - 1\rt] + \bE_{\pih}\lf[S_0 \mid s(0) = s\rt] \notag \\
    & \qquad \qquad + \bE_{\pih}\lf[\sum_{j=1}^{J_{\ve}-1} S_j\rt]. \label{eq:k(s)}
    }
    Now, it follows from Lemma~\ref{lemma:drift} that the first term on the right hand of the above equation is finite. The proof will complete if we show that the second and third terms in the right hand side of the above equation are also finite. This is done next. 

    a) \textbf{Bounding $\bE_{\pih}\lf[\sum_{j=1}^{J_{\ve}-1} S_j\rt]$}: Since $\{S_j\}_{j \ge 1}$ is an i.i.d. sequence and $J_{\ve}$ is a stopping time with $\bE[J_{\ve}] = 1/q$, it follows from Wald's equation~\cite[Theorem 2.6.2]{durrett2019probability} that, 
    \al{
    & \bE_{\pih}\lf[\sum_{j=1}^{J_{\ve}-1} S_j \rt] = \bE_{\pih}[S_1] \bE[J_{\ve}-1] \\
    & \le \bE_{\pih}[S_1] \bE[J_{\ve}]  = \frac{\bE_{\pih}[S_1]}{q}. \label{eq:finite_1}
    }
    We now bound the term $\bE_{\pih}[S_1]$. For this, we note that within $E_1$, the error resets at $\varsigma_1 +1$ with $\D(\varsigma_1 + 1) = w_1 \sim \cN(0,1)$, and then evolves under consecutive transmission failures. Thus, we can write $\D(\varsigma_1 + 1+i) = a^i w_1 + \sum_{l= 1}^{i} a^{i-1} w(\varsigma_1 + l)$, for $i \ge 0$. Taking the expectation of the squared error yields,
    \nal{
    & \bE_{\pih} \lf[\D(\varsigma_1 + 1 + i)^2\rt] = \lf(a^{2i} + \sum_{m=0}^{i-1} a^{2m}\rt).
    }
    Now at $t = \varsigma_1+1$, we have $c(\varsigma_1) = 1$ followed by $c(t) = 0$ for all $\varsigma_1+1 \le t \le \varsigma_{2}$. This implies that $\bP(\varsigma_2 - \varsigma_1 > i) = p_{10} (1-p_{01})^{i-1}$ for $i \ge 1$, and $\bP(\varsigma_2 - \varsigma_1 > 0) = 1$. As a result, we have, 
    \nal{
    & \bE_{\pih}[S_1] = \sum_{t \in E_1} \bE_{\pih}[\D(t)^2] \\
    & = \sum_{i=0}^{\varsigma_2 - \varsigma_1 - 1} \bE_{\pih} \lf[\D(\varsigma_1 + 1 + i)^2\rt] \\
    & =\sum_{i = 0}^{\infty} \bE_{\pih}\lf[\D(\varsigma_1 + 1 + i)^2 \mathds{1}(\varsigma_2 - \varsigma_1 > i)\rt] \\
    & = \sum_{i=0}^{\infty} \bE_{\pih} \lf[\D(\varsigma_1 + 1 + i)^2\rt] \bP(\varsigma_2 - \varsigma_1 > i) \\
    & = 1 + \sum_{i=1}^{\infty} \lf(a^{2i} + \sum_{m=0}^{i-1} a^{2m}\rt) p_{10} (1-p_{01})^{i-1} \\
    & = 1 + \frac{a^2 p_{10}}{1- a^2(1-p_{01})} + p_{10} \sum_{m=0}^{\infty} a^{2m} \sum_{i=m+1}^{\infty} (1-p_{01})^{i-1} \\
    & = 1 + \frac{p_{10}}{1-a^2(1-p_{01})} \lf(a^2 + \frac{1}{p_{01}}\rt) \\
    & < \infty,
    }
    where the fourth equality follows since the event $\{\varsigma_2 - \varsigma > i\}$ depends on Channel 1 process $\{c(t)\}$ that is independent of the noise process $\{w(t)\}$ and hence, independent of $\D(\varsigma_1 + 1 +i)$. The second last equality follows from Assumption~\ref{ass:sys_chn}. Combining the above result with~\eqref{eq:finite_1} we have that,
    \al{
    \bE_{\pih}\lf[\sum_{j=1}^{J_{\ve}-1} S_j\rt] < \infty. \label{eq:bound_a}
    }

    b) \textbf{Bounding $\bE_{\pih} \lf[S_0 \mid s(0) = s\rt]$}: For any initial state $s=(\D_0,0,c_0)$ we have that $\bP(\varsigma_1 > i) = p_{c_0 0} (1 - p_{01})^i$ for $i \ge 1$, and on $\{\varsigma_1 > i\}$, we have that $\D(i) = a^{i} \D_0 + \sum_{l=1}^{i-1} a^{i-1-l} w(l)$. Hence, following on similar lines as part a) for bounding $\bE_{\pih}[S_1]$ we have that under Assumption~\ref{ass:sys_chn},
    \al{
     & \bE_{\pih} \lf[S_0 \mid s(0) = s\rt] \notag \\
     & = \lf(1 + \frac{a^2 p_{c_0 0}}{1 - a^2(1-p_{01})}\rt) \D_0^2 + \frac{p_{c_0 0}}{p_{01}(1 - a^2(1-p_{01}))}. \notag \\
     & < \infty. \label{eq:bound_b}
    }
    Finally, combining~\eqref{eq:bound_a} and~\eqref{eq:bound_b} with~\eqref{eq:k(s)} completes the proof.
\end{proof}

\begin{lemma} \label{lemma:drift}
    
Consider the Markov process $\{(\D(t),r(t), c(t-1))\}$ induced by the policy $\pi^{(hyb.)}$, and let $\ta:=\min\{t\ge 1: s(t)\in\cA(\ve)\}$ with $\cA(\varepsilon) := \{(\D,r,c) \in \cS: |\D| \le \varepsilon, r = 0, c = 1\}$. Then there exist a function $W:\cS\rightarrow[0,\infty)$ and a finite constant $b < \infty$ such that
\nal{
\bE_{\pih}[W(s(1)) - W(s(0)) \mid s(0) = s] \le -1 + b \mathds{1}_{\cA(\ve)}(s).
}

Consequently, we have
\nal{
\bE_{\pih}[\ta \mid s(0) = s] \le W(s) + b\mathds{1}_{\cA(\ve)}(s) < \infty.
}
\end{lemma}

\begin{proof}
    We first note that,
    \al{
    \bP(s(1) \in \cA(\ve) \mid s(0) = (\D, 0, c)) = p_{c1} \bP(|w| \le \ve),~\label{eq:drift_1} 
    }
    where $\{w(t)\}$ is the noise process in~\eqref{source}. Let $q:= \bP(|w| \le \ve)$ and $\underline{\delta} := \min\{p_{01}, 1-p_{01}\}q$. Then, we have from~\eqref{eq:drift_1} that $\bP(s(1) \in \cA(\ve) \mid s(0) = (\D, 0, c)) \ge \underline{\delta}$. Also, since $p_{01} \in (0,1)$ and $(1-p_{01}) \in (0,1)$, we have $\underline{\delta} \in (0,1)$. Hence, $1/\underline{\delta} < \infty$. Now, we define a function $W: \cS \rightarrow [0, \infty]$ as,
    \al{ \label{eq:lyap}
    W(\D,r,c) =
    \begin{cases}
        0 \mbox{ if } (\D,r,c) \in \cA(\ve), \\
        \frac{1}{\underline{\delta}} \mbox{ if } r=0, (\D,r,c) \neq \cA(\ve), \\
        \frac{1}{\underline{\delta}} + r \mbox{ if } r \ge 1.
    \end{cases}
    }
    Thus, we have that $0 \le W \le 1/\underline{\delta} + (d-1)$. 

    Now, we show that $\bE_{\pih}[W(s(1)) - W(s(0)) \mid s(0) = (\D,r,c)] \le -1$ for every $(\D,r,c) \notin \cA(\ve)$. We consider the following cases:

    Case i) $r = 0, c=0$. With $u = 1$, it follows from Table~\ref{tab:df} that with probability $1-p_{01}$, the next state is $(a\D+w,0,0) \notin \cA(\ve)$. So, we have from~\eqref{eq:lyap} that $W(a\D+w,0,0) = 1/\underline{\delta}$. Moreover, with probability $p_{01}$, the next state is $(w,0,1) \in \cA(\ve)$ if and only if $|w| \le \ve$.
     Hence, we have
     \nal{
     & \bE_{\pih}[W(s(1)) - W(s(0)) \mid s(0) = (\D,r,c)] \\
     & = (1 - p_{01}) \frac{1}{\underline{\delta}} + p_{01} (1- q) \frac{1}{\underline{\delta}} - \frac{1}{\underline{\delta}} \\
     & = -\frac{p_{01}q}{\underline{\delta}} \le -1,
     }
     where the last inequality follows since $p_{01}q \ge \underline{\delta}$ by definition of $\underline{\delta}$.

     Case ii) $r = 0, c = 1, |\D| > \ve$. With $u=1$, it follows from Table~\ref{tab:df} that with probability $p_{10}$, the next state is $(a\D+w,0,0) \notin \cA(\ve)$. So, we have from~\eqref{eq:lyap} that $W(a\D+w,0,0) = 1/\underline{\delta}$. Moreover, with probability $1-p_{10}$, the next state is $(w,0,1) \in \cA(\ve)$ if and only if $|w| \le \ve$.
     Hence, we have
     \nal{
     & \bE_{\pih}[W(s(1)) - W(s(0)) \mid s(0) = (\D,r,c)] \\
     & = -\frac{1-p_{10}q}{\underline{\delta}} \le -1,
     }
     where the last inequality follows since $p_{01}q \ge \underline{\delta}$ by definition of $\underline{\delta}$. 

     Case iii) $r \ge 2$. For this case, we have $u = 0$. Then, the next state is $(a\D+w, r-1, c_+) \notin \cA(\ve)$. Hence, we have
     \nal{
     & \bE_{\pih}[W(s(1)) - W(s(0)) \mid s(0) = (\D,r,c)] \\
     & = \frac{1}{\underline{\delta}} + (r-1) - \frac{1}{\underline{\delta}} + r = -1.
     }

     Case iv) $r = 1$. For this case, we have $u = 0$. Then, it follows from Table~\ref{tab:df} that the next state is $(\D_+,0,c_+)$ such that $\D_+ \sim \cN(0, \sigma^2_d)$ and $c_+ \in \{0,1\}$. Hence, $\bE_{\pih}[W(s(1)) \mid s(0) = (\D,r,c)] \in \{0,1/\underline{\delta}\}$ and $W(s(0)) = 1/\underline{\delta} + 1$. Hence, $\bE_{\pih}[W(s(1)) - W(s(0)) \mid s(0) = (\D,r,c)] \le -1$.

    Next, we show that $\bE_{\pih}[W(s(1)) - W(s(0)) \mid s(0) = (\D,r,c)] \le -1 + b$ for some $b < \infty$ and every $(\D,r,c) \in \cA(\ve)$. We observe that for all states with $r=0$, it follows from~\eqref{eq:lyap} that $\bE_{\pih}[W(s(1)) \mid s(0) = (\D,r,c)] \le \frac{1}{\underline{\delta}}$. Thus, $\bE_{\pih}[W(s(1)) - W(s(0)) \mid s(0) = (\D,r,c)] \le \frac{1}{\underline{\delta}}$ since $W(s(0)) = 0$ from~\eqref{eq:lyap}. The proof that $\bE_{\pih}[W(s(1)) - W(s(0)) \mid s(0) = s] \le -1 + b \mathds{1}_{\cA(\ve)}(s)$ completes by taking $b = 1/\underline{\delta} + 1$.

    It follows from~\cite[Theorem 11.3.4]{meyn2012markov} that, 
    \nal{
    \bE_{\pih}[\ta \mid s(0) = s] \le W(s) + b\mathds{1}_{\cA(\ve)}(s).
    }
    The finiteness of $\bE_{\pih}[\ta \mid s(0) = s]$ follows since both $W$ and $b$ are finite. This completes the proof.
\end{proof}

\bibliographystyle{IEEEtran}
\bibliography{refs_v2}

\end{document}